\documentclass[11pt]{article}

\usepackage{mathpazo}
\usepackage{mathdots}
\usepackage{amsfonts}
\usepackage{amsmath}
\usepackage{graphicx}
\usepackage{latexsym}
\usepackage{mathabx}
\usepackage{bm}
\usepackage{longtable}
\usepackage{booktabs}
\usepackage{makecell}

\usepackage[margin=1.05in]{geometry}
  
\usepackage[T1]{fontenc}
\usepackage{times}
\usepackage{array}
\usepackage{enumitem}
\usepackage{amssymb}
\usepackage{pgfplots}
\usepackage{pgf}
\usepackage{tikz}
\usetikzlibrary{patterns}
\usepgfplotslibrary{patchplots} 
\usetikzlibrary{decorations.markings}
\usetikzlibrary{pgfplots.patchplots} 
\pgfplotsset{width=9cm,compat=1.12}
    \usepgfplotslibrary{colormaps,external}
    \definecolor{ocre}{RGB}{0,96,128}
    \definecolor{purp}{RGB}{112,0,112}
    \pgfplotsset{
    	colormap={ocrefade}{
    		rgb255=(150,216,255)
    		rgb255=(0,64,96)
    	}
    }

\usepackage{amsthm}

\newtheorem{definition}{Definition}
\newtheorem{lemma}{Lemma}
\newtheorem{theorem}{Theorem}
\newtheorem{corollary}{Corollary}

\newtheorem{proposition}{Proposition}

\newtheorem{example}{Example}

\usepackage[dvipsnames]{xcolor}
\usepackage{makeidx}
\usepackage[colorlinks=true,linkcolor=blue,anchorcolor=blue,citecolor=red,urlcolor=magenta]{hyperref}
\usepackage{caption}
\def\R{\mathbb{R}}

\def\Zp{\mathbb{Z}_{> 0}}
\def\Znn{\mathbb{Z}_{\geq 0}}
\def\vv{\mathbf{v}} 

\begin{document}

\title{Enumeration of Laplacian integral and \\ $\{-1,0,1\}$-diagonalizable graphs}

\author{ 
	Nathaniel Johnston\textsuperscript{1},  Sarah Plosker\textsuperscript{2}, and Luis M.~B.~Varona\textsuperscript{1,3,4}
}

\maketitle

\begin{abstract}
    A graph with Laplacian matrix $L$ is called Laplacian integral if the eigenvalues of $L$ are all integers, and it is called $\{-1,0,1\}$-diagonalizable if $L$ has a full set of eigenvectors with entries from $\{-1,0,1\}$. We herein develop a structure theorem for both Laplacian integral graphs and $\{-1,0,1\}$-diagonalizable graphs of prime order, and combine it with some novel computational techniques to characterize all such graphs for orders larger than was previously possible. For example, we enumerate all Laplacian integral and $\{-1,0,1\}$-diagonalizable graphs of order $13$ or less, all $\{-1,0,1\}$-diagonalizable graphs of prime order $23$ or less, all regular integral graphs of order $15$ or less, and all regular $\{-1,0,1\}$-diagonalizable graphs of prime order $53$ or less. As an immediate byproduct of our work, we show that the $S_{n,n}$ conjecture for Laplacian integral graphs is true when $n = 12$, thus making $n = 16$ the smallest open case; additionally, we disprove two related conjectures regarding Laplacian spectra. We also establish an exponential lower bound on the number of connected $\{-1,0,1\}$-diagonalizable graphs of order $n$, thus beating the previously best-known (subexponential) lower bound. Finally, we show that every bipartite $\{-1,0,1\}$-diagonalizable graph is regular (a fact that fails to generalize to Laplacian integral graphs).\\

	
\end{abstract}

\addtocounter{footnote}{1}
\footnotetext{Department of Mathematics \& Computer Science, Mount Allison University, Sackville, NB, Canada E4L 1E4}
\addtocounter{footnote}{1}
\footnotetext{Department of Mathematics \& Computer Science, Brandon University, Brandon, MB, Canada R7A 6A9}
\addtocounter{footnote}{1}
\footnotetext{Department of Politics \& International Relations, Mount Allison University, Sackville, NB, Canada E4L 1E4}
\addtocounter{footnote}{1}
\footnotetext{Department of Economics, Mount Allison University, Sackville, NB, Canada E4L 1E4}

\section{Introduction}

The study of the spectral properties of a matrix associated with a graph is the central focus of the field of spectral graph theory. An \emph{integral} graph is one for which all of the eigenvalues of its adjacency matrix are integers. Integral graphs have been studied since the 1970s \cite{harary2006graphs}; see \cite{balinska2002survey, wang2005survey} for surveys on the topic. Similarly, a \emph{Laplacian integral} graph is one for which all of the eigenvalues of its Laplacian matrix are integers. Laplacian integral graphs have been studied since the 1990s \cite{grone1990laplacian,GM94,merris1994degree,merris1994laplacian}, though other spectral properties of the Laplacian matrix were investigated before that \cite{fiedler1973algebraic}. There has recently been increased interest in exploring not only these graphs but also other related families because of their relationships to continuous-time quantum walks (CTQWs), which model information flow in quantum networks.

As a refinement of Laplacian integrality, the eigenvectors of Laplacian integral graphs present another structure often of interest \cite{Mer98,MN03,CKK19}. For example, graphs with an orthogonal basis of eigenvectors with entries from $\{-1,1\}$ are called \emph{Hadamard-diagonalizable} \cite{BFK11,Breen22}; one can readily verify that such graphs are necessarily Laplacian integral. Several other variants that have been explored include weakly Hadamard-diagonalizable graphs \cite{adm2021weakly} (which slightly relax the orthogonality requirement and allow entries from $\{-1,0,1\}$ more broadly) as well as $\{-1,0,1\}$-diagonalizable and $\{-1,1\}$-diagonalizable graphs \cite{johnston2025laplacian} (which remove the orthogonality requirement entirely).

These families of graphs connect to CTQWs in many ways---for example, it was shown in \cite{God11} that a regular graph is periodic (roughly, the flow of quantum information in the network repeats periodically) if and only if it is integral.\footnote{For regular graphs, integrality is equivalent to Laplacian integrality.} It was shown in \cite{johnston2017} that Hadamard-diagonalizable graphs are particularly conducive to perfect state transfer (roughly, quantum information flows completely from one vertex to another), and it was shown in \cite{MMP25} that the same is true of weakly Hadamard-diagonalizable graphs. Most recently, $\{-1,0,1\}$-diagonalizable graphs have been shown to be particularly conducive to pair state transfer (see \cite{Mon26}, though $\{-1,0,1\}$-diagonalizable graphs were called \emph{simply structured} graphs there).

In this work, we focus on the families of Laplacian integral graphs and $\{-1,0,1\}$-diagonalizable graphs. We present a structure theorem for such graphs of prime order which, when combined with extensive computations (see Appendix~\ref{sec:appendix_compute_lapint}), lets us extend the enumeration of all Laplacian integral and $\{-1,0,1\}$-diagonalizable graphs to larger orders than was previously possible. In particular, our most notable computational results include:
\begin{itemize}
    \item We enumerate all Laplacian integral graphs of order $13$ or less (previously only known in the connected case up to order $11$ \cite{oeisA363064} and in the general case up to order $10$ \cite{oeisA363065}) and all $\{-1,0,1\}$-diagonalizable graphs of order $13$ or less (previously only known up to order $9$ \cite[Table~3]{johnston2025laplacian}). We also enumerate all $\{-1,0,1\}$-diagonalizable graphs of orders $17$, $19$, and $23$.
    
    \item We enumerate all regular integral graphs and all regular $\{-1,0,1\}$-diagonalizable graphs of order $15$ or less (previously such regular integral graphs were known up to order $12$ \cite{BKSZ01}, with just a partial enumeration for $n = 13$ \cite{BSZ09}). We also enumerate all regular integral graphs of orders $17$ and $19$, and all regular $\{-1,0,1\}$-diagonalizable graphs of prime order $53$ or less.
    
    \item We enumerate all connected bipartite $\{-1,0,1\}$-diagonalizable graphs of order $16$ or less.
\end{itemize}
Our most notable theoretical contributions include:
\begin{itemize}
    \item We prove a decomposition theorem for prime-order Laplacian integral and $\{-1,0,1\}$-diagonalizable graphs (Theorem~\ref{thm:prime_main}), which makes it trivial to construct all such graphs of prime order whenever all such graphs of smaller orders are known. This result has numerous immediate consequences, such as making the enumeration of order-$13$ Laplacian integral graphs trivial, and explaining why there is no $6$-regular order-$13$ Laplacian integral graph (see Corollary~\ref{cor:reg_lap_prime_existence} and \cite[Problem~3]{BSZ09}).
    
    \item As an immediate corollary of our enumeration of Laplacian integral graphs, we extend a conjecture on Laplacian integral graphs that is known to be true for $n \leq 11$ vertices to $n \leq 15$. In particular, in \cite{fallat2005graphs}, it was conjectured that there exists no graph on $n$ vertices whose Laplacian spectrum is $\{0,1,\ldots, n-1\}$ (this is called the \emph{$S_{n,n}$ conjecture}). This conjecture was proved for all graphs on $n$ vertices when $n$ is prime or $n \equiv 2,3 \pmod{4}$, and it was verified computationally for all graphs up to order $11$ and for all self-complementary graphs up to order $12$ \cite{HKT22}. Our enumeration of order $n = 12$ Laplacian integral graphs shows that the $S_{n,n}$ conjecture is true (without the self-complementarity assumption) in that case as well, leaving $n = 16$ as the smallest open case.

    \item Our enumeration results additionally disprove multiple conjectures from~\cite{HT23,HT25}. Those papers define
    \[
        S_{\{i,j\}_n^m} = \{0,1,2,\ldots,m-1,m,m,m+1,\ldots,n-1,n\} \setminus \{i,j\}
    \]
    (i.e., the set of non-negative integers no larger than $n$, with $m$ occurring twice, excluding $i$ and $j$).

    Conjecture~5.4 of \cite{HT23} says that if $n \geq 9$ then no spectra of the form $S_{\{i,n\}_n^m}$ are realized by Laplacian matrices of graphs. Our computational results show that this is false, and there are exactly two counterexamples when $n \in \{9,10,\ldots,13\}$: the spectrum $S_{\{8,9\}_9^6}$ is realized by the Laplacian matrix of the graph in Figure~\ref{fig:Sn2_counterexample}, and the spectrum $S_{\{1,9\}_9^3}$ is attained by its complement.

    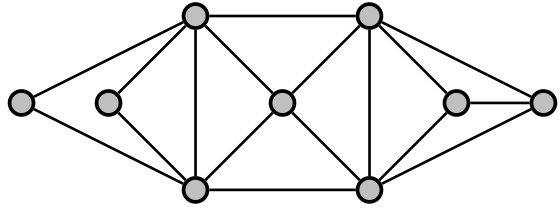
\begin{figure}[!htb]
        \centering
        \begin{tikzpicture}[scale=1.15,every node/.style={circle,draw=black,fill=gray!50,line width=1.4pt,minimum size=9pt,inner sep=2pt},every path/.style={draw=black,line width=1pt}]
        
        \node (a) at (1,2) {};        
        \node (b) at (1,0) {};        
        \node (c) at (-1,1) {};   
        \node (d) at (0,1) {};   
        \node (u) at (2,1) {};    
        
        \node (e) at (3,2) {};    
        \node (f) at (3,0) {};   
        \node (g) at (4,1) {};    
        \node (h) at (5,1) {};    
        
        \draw (a) -- (e);
        \draw (e) -- (f);
        \draw (f) -- (b);
        \draw (b) -- (a);
        
        \draw (u) -- (a);
        \draw (u) -- (b);
        \draw (u) -- (e);
        \draw (u) -- (f);
        
        \draw (c) -- (a);
        \draw (c) -- (b);
        \draw (d) -- (a);
        \draw (d) -- (b);
        
        \draw (g) -- (e);
        \draw (g) -- (f);
        \draw (h) -- (e);
        \draw (h) -- (f);
        \draw (g) -- (h);
        
        \end{tikzpicture}
        \caption{A graph with spectrum $S_{\{8,9\}_9^6}$, thus providing a counterexample to \cite[Conjecture~5.4]{HT23}.}\label{fig:Sn2_counterexample}
    \end{figure}
    
    Similarly, Conjecture~5.1 of \cite{HT25} says that if $n \geq 6$ then no spectra of the form $S_{\{1,j\}_n^2}$ are realized by Laplacian matrices of graphs. Our computational results show that this is false, and there are exactly three counterexamples when $n \in \{6,7,\ldots,12\}$:
    \begin{itemize}
        \item The spectrum $S_{\{1,7\}_8^2}$ is realized by the graph $(K_2 \sqcup (K_1 \sqcup K_2 \sqcup K_{1,2})^{\textup{c}})^{\textup{c}}$.

        \item The spectrum $S_{\{1,8\}_9^2}$ is realized by the graph $(K_2 \sqcup (K_1 \sqcup K_2 \sqcup W)^{\textup{c}})^{\textup{c}}$, where $W$ is the $4$-vertex paw graph.

        \item The spectrum $S_{\{1,11\}_{12}^2}$ is realized by the graph $(K_2 \sqcup (K_1 \sqcup K_2 \sqcup J)^{\textup{c}})^{\textup{c}}$, where $J$ is the $7$-vertex graph $J = (K_1 \sqcup (K_1 \sqcup (K_2 \sqcup K_{1,2})^{\textup{c}})^{\textup{c}})^{\textup{c}}$.
    \end{itemize}

    \item We prove an exponential lower bound on the number of connected $\{-1,0,1\}$-diagonalizable graphs (Theorem~\ref{thm:nozo_lower_bound}), thus improving upon the known subexponential lower bound of \cite[Corollary~4]{johnston2025laplacian}.
    
    \item We prove that every connected bipartite $\{-1,0,1\}$-diagonalizable graph is regular (Theorem~\ref{thm:bipartite_imp_regular}).
\end{itemize}

The remainder of the paper is organized as follows. In Section~\ref{sec:notation}, we introduce our notation and the background material needed throughout. In particular, Section~\ref{sec:balanced} is devoted to the concept of balanced vectors, which are one of the main tools that we use to characterize $\{-1,0,1\}$-diagonalizable graphs. In Section~\ref{sec:prime}, we present and prove our structure theorem for Laplacian integral and $\{-1,0,1\}$-diagonalizable graphs of prime order. This helps us enumerate Laplacian integral graphs and $\{-1,0,1\}$-diagonalizable graphs in Sections~\ref{sec:counting_lap} and~\ref{sec:counting_zero_one}, respectively.

In Section~\ref{sec:prime_regular}, we explore what these results tell us if we restrict our attention to \emph{regular} integral graphs, and we do the same for regular $\{-1,0,1\}$-diagonalizable graphs in Section~\ref{sec:prime_regular_negone}. In Section~\ref{sec:asymp}, we show that if $n$ is large then there are exponentially many connected $\{-1,0,1\}$-diagonalizable graphs. In Section~\ref{sec:bipartite}, we focus on connected bipartite $\{-1,0,1\}$-diagonalizable graphs, and in particular we show that every such graph is regular. We conclude with some open problems in Section~\ref{sec:conclude}. Finally, we describe in the appendix our computational methods for determining Laplacian integrality and $\{-1,0,1\}$-diagonalizability (code and raw data are available at \cite{VaronaCode}).

\section{Notation and preliminaries}\label{sec:notation}

We use bold lowercase letters like $\vv \in \mathbb R^d$ to denote a real-valued $d$-dimensional vector and subscripts like $v_j$ or $[\vv]_j$ to denote the $j$-th entry of $\vv$. Let $\mathbf{1}_d$ and $\mathbf{0}_d$ denote the all-ones and all-zeros vectors, respectively, or simply $\mathbf{1}$ and $\mathbf{0}$ when the dimension is understood. We use $\mathcal{M}_{m,n}$ to denote the set of $m \times n$ real matrices, with $\mathcal{M}_n$ being used when $m=n$. Let $I_n$, $J_n$, and $O_n$ denote the $n\times n$ identity matrix, all-ones matrix, and all-zeros matrix, respectively, or simply $I$, $J$, or $O$ when the dimension is understood. Given a matrix $A \in \mathcal{M}_{m,n}$, its $(i,j)$-th entry is denoted $a_{i,j}$ or $[A]_{i,j}$. 

Let $G$ be a simple graph---an unweighted, undirected graph with no loops---with vertex and edge sets $V(G)$ and $E(G)$, respectively. Let $G^\textup{c}$ be the complement of $G$. Let $A(G)$, $D(G)$, and $L(G):=D(G)-A(G)$ be the adjacency matrix, diagonal degree matrix, and Laplacian matrix of $G$, respectively. Let $\mathrm{Lin}(G)$ be the line graph of $G$. We use $\sqcup$ to denote the disjoint union of graphs, and $G \mathbin{\square} H$ for the Cartesian product of $G$ and $H$. We use the notation $kG = G \sqcup G \sqcup \cdots \sqcup G$ to denote the disjoint union of $k$ copies of $G$. A complete multipartite graph $K_{n_1, n_2, \dots, n_d}$ is a $d$-partite graph on $n$ vertices in which there is an edge between every pair of vertices from the $d$ different independent sets of sizes $n_1, n_2, \dots, n_d$, where $\sum_{j=1}^d n_j=n$. We use $C_n$ to denote the cycle graph of order $n$ (the ``order'' of a graph $G$ is its number of vertices, denoted by $n=|V(G)|$), and $K_n$ for the complete graph of order $n$.

We now recall \cite[Definition 1]{johnston2025laplacian}:

\begin{definition}\label{defn:negoneone_diag}
    A graph $G$ is \emph{$\{-1,0,1\}$-diagonalizable} if there is a basis of eigenvectors for $L(G)$ whose entries all belong to $\{-1,0,1\}$. Equivalently, $G$ is $\{-1,0,1\}$-diagonalizable if there exists an invertible matrix $P$ whose entries all belong to $\{-1,0,1\}$ with the property that $P^{-1}L(G)P$ is diagonal.
\end{definition}
 
It is straightforward to show that every $\{-1,0,1\}$-diagonalizable graph is Laplacian integral \cite{johnston2025laplacian}. One of our main tasks is to enumerate Laplacian integral and $\{-1,0,1\}$-diagonalizable graphs of small orders. For this purpose, it will be convenient to give names to the functions that count these graphs:
\begin{itemize}
    \item $l(n)$: the number of Laplacian integral graphs of order $n$.
    
    \item $cl(n)$: the number of connected Laplacian integral graphs of order $n$.
    
    \item $s(n)$: the number of $\{-1,0,1\}$-diagonalizable graphs of order $n$.\footnote{We use ``s'' for the name of this function in reference to the name ``structurally simple'' that was given to these graphs in \cite{Mon26}.}
    
    \item $cs(n)$: the number of connected $\{-1,0,1\}$-diagonalizable graphs of order $n$.
\end{itemize}

Finally, we will use the notation $m_p(k)$ for the number of occurrences of $k$ in the multiset $p\in P_n$, where $P_n$ is the set of integer partitions of $n$. For example, if $n = 4$ then $p = \{2,1,1\}$ is an integer partition of $4$, and
\[
    P_4 = \big\{ \{1,1,1,1\}, \ \{2,1,1\}, \ \{2,2\}, \ \{3,1\}, \ \{4\}\big\}.
\]

\subsection{Balanced vectors}\label{sec:balanced}

The concept of ``balanced'' vectors plays a key role in the theory of $\{-1,0,1\}$-diagonalizable graphs \cite{johnston2025laplacian}. Here, we recall the definition and some basic results concerning these vectors:

\begin{definition}\label{defn:balanced}
    A vector $\vv \in \R^d$ ($d \geq 2$) is \emph{balanced} if there exists a matrix $A \in \mathcal{M}_{d-1,d}$, all of whose entries belong to $\{-1,0,1\}$, for which $\mathrm{null}(A) = \mathrm{span}(\vv)$. Every $\vv \in \R^d$ for $d = 1$ is balanced.
\end{definition}

By \cite[Fact~1]{johnston2025laplacian}, every balanced vector can have its entries scaled and permuted, and its signs changed (without changing balancedness), so it suffices to consider balanced vectors $\vv \in \Zp^d$ that have their entries sorted in non-increasing order (equivalently, we can associate balanced vectors with balanced \emph{multisets} of positive integers). We make these assumptions about balanced vectors throughout the remainder of the paper. We now prove some basic results about balanced vectors that we will need later.

\begin{lemma}\label{lem:balanced_max_entry}
    Let $n,d \in \Zp$ ($d \geq 2$) and suppose $\vv = (v_1, v_2, \ldots, v_d) \in \Zp^d$ is balanced and has
    \[
        v_1 + v_2 + \cdots + v_d = n.
    \]
    Then $\displaystyle \max_{1 \leq j \leq d}\{v_j\} \leq \lfloor n/2 \rfloor$.
\end{lemma}

\begin{proof}
    Suppose (for the sake of establishing a contradiction) that there exists an index $j$ for which $v_j \geq \lfloor n/2 \rfloor + 1$. Because $v_1 + v_2 + \cdots + v_d = n$, we have
    \begin{align}\label{ineq:vi_sum}
        \sum_{i \neq j} v_i \leq n - \lfloor n/2 \rfloor - 1.
    \end{align}
    In particular, this implies that for every $\mathbf{a} \in \{-1,0,1\}^d$ with $a_j \neq 0$, we have
    \begin{align*}
        |\mathbf{a} \cdot \mathbf{v}| & = \left|\sum_{i=1}^d a_i v_i\right| \\
        & \geq v_j - \left|\sum_{i \neq j} a_i v_i\right| \\
        & \geq v_j - \sum_{i \neq j} v_i \\
        & \geq 2 + 2\lfloor n/2 \rfloor -  n,
    \end{align*}
    where the first inequality follows from the reverse triangle inequality, the second inequality follows from the standard triangle inequality along with the fact that $v_i \geq 0$ for all $i$, and the third inequality follows from Inequality~\eqref{ineq:vi_sum}.

    Since $2 + 2\lfloor n/2 \rfloor -  n \geq 1$ for all $n \in \Zp$, we conclude that $\mathbf{a} \cdot \mathbf{v} = 0$ is impossible, so the only vectors $\mathbf{a} \in \{-1,0,1\}^d$ for which $\mathbf{a} \cdot \mathbf{v} = 0$ are those with $a_j = 0$. It follows that if $A \in \mathcal{M}_{d-1,d}$ is such that $A\vv = \mathbf{0}$ then the $j$-th column of $A$ consists entirely of zeroes. Since $\mathrm{null}(A)$ is one-dimensional and contains the vector with a $1$ in its $j$-th entry and $0$ elsewhere, $\vv$ must be a multiple of that vector. This contradicts the assumption that all entries of $\vv$ are strictly positive and completes the proof.
\end{proof}

We will also need some results that characterize which vectors with just a few distinct entries are balanced. When there is just one distinct entry (i.e., $\mathbf{v} = c\mathbf{1}_d$ for some $c,d \in \Zp$), $\mathbf{v}$ is always balanced, since the null space of the $(d-1) \times d$ matrix
\begin{align}\label{eq:balanced_A_all_ones}
    \begin{bmatrix}
        1 & 0 & 0 & \cdots & 0 & -1 \\
        0 & 1 & 0 & \cdots & 0 & -1 \\
        0 & 0 & 1 & \cdots & 0 & -1 \\
        \vdots & \vdots & \vdots & \ddots & \vdots & \vdots \\
        0 & 0 & 0 & \cdots & 1 & -1 \\
    \end{bmatrix}
\end{align}
is exactly $\mathrm{span}(\mathbf{1}_d)$. When $\mathbf{v}$ has two distinct entries, it is no longer the case that $\mathbf{v}$ is always balanced, but it is still straightforward to determine whether it is:

\begin{proposition}\label{prop:ab_balanced}
    Suppose $a,b,n_a,n_b \in \Zp$ and let $d = \gcd(a,b)$. Then $(a\mathbf{1}_{n_a},b\mathbf{1}_{n_b})$ is balanced if and only if
    \[
        n_a \geq \frac{b}{d} \quad \text{and} \quad n_b \geq \frac{a}{d}.
    \]
\end{proposition}

\begin{proof}
    Let $\mathbf{v} = (a\mathbf{1}_{n_a},b\mathbf{1}_{n_b})$ and notice that $\mathbf{v}$ is balanced if and only if $\mathbf{v} / d$ is balanced. We can (and will) thus assume without loss of generality that $d = 1$.

    For the ``if'' direction of the proof, recall from \cite[Proposition~3]{johnston2025laplacian} that if there exist vectors $\mathbf{x} \in \{-1,0,1\}^{n_a}$ and $\mathbf{y} \in \{-1,0,1\}^{n_b}$ for which $(a\mathbf{1}_{n_a}) \cdot \mathbf{x} = (b\mathbf{1}_{n_b}) \cdot \mathbf{y} \neq 0$ then $\mathbf{v} = (a\mathbf{1}_{n_a},b\mathbf{1}_{n_b})$ is balanced. Since $n_a \geq b$ and $n_b \geq a$, we can choose $\mathbf{x} = (\mathbf{1}_{b},\mathbf{0}_{n_a-b})$ and $\mathbf{y} = (\mathbf{1}_{a},\mathbf{0}_{n_b-a})$ to get
    \[
        (a\mathbf{1}_{n_a}) \cdot \mathbf{x} = (b\mathbf{1}_{n_b}) \cdot \mathbf{y} = ab \neq 0.
    \]
    It follows that $\mathbf{v}$ is balanced.

    For the ``only if'' direction of the proof, we will prove the contrapositive statement that if $n_a < b$ or $n_b < a$ then $\mathbf{v} = (a\mathbf{1}_{n_a},b\mathbf{1}_{n_b})$ is not balanced. We will assume that $n_a < b$ and simply note that the proof is almost identical if $n_b < a$ instead.
    
    Suppose $A \in \mathcal{M}_{n_a+n_b-1,n_a+n_b}$ has entries from $\{-1,0,1\}$ and is such that $A\mathbf{v} = \mathbf{0}$. Our goal is to show that $\mathrm{null}(A)$ is at least $2$-dimensional, so that $\mathrm{null}(A) \neq \mathrm{span}(\mathbf{v})$. To this end, let $\mathbf{w} \in \{-1,0,1\}^{n_a + n_b}$ be any row of $A$, so that $\mathbf{w} \cdot \mathbf{v} = 0$. If we write $\mathbf{w} = (\mathbf{a},\mathbf{b})$ for some $\mathbf{a} \in \{-1,0,1\}^{n_a}$ and $\mathbf{b} \in \{-1,0,1\}^{n_b}$ then this tells us that
    \begin{align}\label{eq:ab_ip_zero}
        0 = (\mathbf{a},\mathbf{b}) \cdot (a\mathbf{1}_{n_a},b\mathbf{1}_{n_b}) = as_a + bs_b,
    \end{align}
    where $s_a$ is the sum of the entries of $\mathbf{a}$ (notice that $|s_a| \leq n_a$) and $s_b$ is the sum of the entries of $\mathbf{b}$.
    
    Computing Equation~\eqref{eq:ab_ip_zero} modulo $b$ and using the assumption that $\gcd(a,b) = 1$ shows that $s_a$ is an integer multiple of $b$. Since $|s_a| \leq n_a < b$, the only way that $s_a$ can be an integer multiple of $b$ is if $s_a = 0$, which then (via Equation~\eqref{eq:ab_ip_zero}) implies $bs_b = 0$, so $s_b = 0$ too. Since this argument applies to every row of $A$, we see that the rows of $A$ are contained in the $((n_a - 1) + (n_b - 1))$-dimensional subspace
    \[
        \big\{ (\mathbf{a}, \mathbf{b}) : \mathbf{a} \in \R^{n_a}, \mathbf{b} \in \R^{n_b}, \mathbf{a} \cdot \mathbf{1}_{n_a} = \mathbf{b} \cdot \mathbf{1}_{n_b} = 0 \big\}
    \]
    of $\R^{n_a + n_b}$. It follows that $A$ has rank at most $(n_a - 1) + (n_b - 1) = n_a + n_b - 2$, so (by the rank--nullity theorem) $\dim(\mathrm{null}(A)) \geq 2$, which completes the proof.
\end{proof}

When $\mathbf{v}$ has three or more distinct entries, determining balancedness is trickier in general. However, we will only need the special case when those three distinct entries are $3$, $4$, and $6$, in which case the following proposition solves this problem:

\begin{proposition}\label{prop:346_balanced}
    Let $n_3, n_4, n_6 \in \Znn$. Then the vector $(3\mathbf{1}_{n_3},4\mathbf{1}_{n_4},6\mathbf{1}_{n_6})$ is balanced in exactly the following five cases:
    \begin{enumerate}[label = {\alph*)}]
        \item Exactly one of $n_3$, $n_4$, or $n_6$ is non-zero.
        \item $n_3 = 0$, $n_4 \geq 3$, and $n_6 \geq 2$.
        \item $n_3 \geq 2$, $n_4 = 0$, and $n_6 \geq 1$.
        \item $n_3 \geq 4$, $n_4 \geq 3$, and $n_6 = 0$.
        \item $n_3 \geq 2$, $n_4 \geq 3$, and $n_6 \geq 1$.
    \end{enumerate}
\end{proposition}

\begin{proof}
    The first four cases of the proposition follow immediately from Proposition~\ref{prop:ab_balanced}. All that remains is to prove that if $n_3, n_4, n_6 \geq 1$ then $\vv = (3\mathbf{1}_{n_3},4\mathbf{1}_{n_4},6\mathbf{1}_{n_6})$ is balanced if and only if $n_3 \geq 2$ and $n_4 \geq 3$.

    In one direction, suppose that $n_3 \geq 2$ and $n_4 \geq 3$. For $d \in \Zp$, let $B_d$ be the $(d-1) \times d$ matrix~\eqref{eq:balanced_A_all_ones}. Define the following three additional matrices:
    \begin{itemize}
        \item $C \in \mathcal{M}_{2,n_3}$ has $c_{i,j} = 1$ for all $i$ and for $1 \leq j \leq 2$ (which is possible since $n_3 \geq 2$) and $c_{i,j} = 0$ otherwise,
        
        \item $D \in \mathcal{M}_{2,n_4}$ has $d_{1,j} = 0$ for all $j$, $d_{2,j} = -1$ for $1 \leq j \leq 3$ (which is possible since $n_4 \geq 3$), and $d_{i,j} = 0$ otherwise, and
        
        \item $E \in \mathcal{M}_{2,n_6}$ has $e_{1,1} = -1$, $e_{2,1} = 1$, and $e_{i,j} = 0$ otherwise.
    \end{itemize}
    Then straightforward calculation shows that
    \[
        A = \begin{bmatrix}
            B_{n_3} & O & O \\
            O & B_{n_4} & O \\
            O & O & B_{n_6} \\
            C & D & E
        \end{bmatrix}
    \]
    is a $(n_3 + n_4 + n_6 - 1) \times (n_3 + n_4 + n_6)$ matrix with entries from $\{-1,0,1\}$ that has linearly independent rows and $A\vv = \mathbf{0}$. It follows that $\mathrm{null}(A) = \mathrm{span}(\vv)$, so $\vv$ is balanced.

    In the other direction, suppose that $\mathbf{v}$ is balanced and $\mathbf{w}$ is a row of an $(n_3 + n_4 + n_6 - 1) \times (n_3 + n_4 + n_6)$ matrix $A$ with entries from $\{-1,0,1\}$ such that $\mathrm{null}(A) = \mathrm{span}(\vv)$ (so that $\mathbf{w} \cdot \mathbf{v} = 0$). If we write $\mathbf{w} = (\mathbf{a}, \mathbf{b}, \mathbf{c})$ for some $\mathbf{a} \in \{-1,0,1\}^{n_3}$, $\mathbf{b} \in \{-1,0,1\}^{n_4}$, and $\mathbf{c} \in \{-1,0,1\}^{n_6}$, then we have 
    \begin{align}\label{eq:abc_346_sum}
        0 = (\mathbf{a}, \mathbf{b}, \mathbf{c}) \cdot (3\mathbf{1}_{n_3},4\mathbf{1}_{n_4},6\mathbf{1}_{n_6}) = 3s_a + 4s_b + 6s_c,
    \end{align}
    where $s_a$ is the sum of the entries of $\mathbf{a}$ (and similarly for $s_b$ and $s_c$). Notice that $|s_a| \leq n_3$ and $|s_b| \leq n_4$.

    Now assume (for the sake of establishing a contradiction) that $n_3 = 1$ (so $|s_a| \leq 1$). Reducing Equation~\eqref{eq:abc_346_sum} modulo $2$ shows that $s_a$ is even, so $s_a \notin \{-1,1\}$, which forces $s_a = 0$. Plugging this back into Equation~\eqref{eq:abc_346_sum} shows that $4s_b + 6s_c = 0$ as well, so the rows of $A$ are contained in the $((n_3 - 1) + (n_4 + n_6 - 1))$-dimensional subspace
    \[
        \big\{ (\mathbf{a}, \mathbf{b}, \mathbf{c}) : \mathbf{a} \in \R^{n_3}, \mathbf{b} \in \R^{n_4}, \mathbf{c} \in \R^{n_6}, \mathbf{a} \cdot \mathbf{1}_{n_3} = \mathbf{b} \cdot \mathbf{1}_{n_4} + \mathbf{c} \cdot \mathbf{1}_{n_6} = 0 \big\}
    \]
    of $\R^{n_3 + n_4 + n_6}$. It follows that $A$ has rank at most $((n_3 - 1) + (n_4 + n_6 - 1)) = n_3 + n_4 + n_6 - 2$, so (by the rank--nullity theorem) $\dim(\mathrm{null}(A)) \geq 2$, which contradicts the fact that $\vv$ is balanced. This contradiction shows that $n_3 \geq 2$.

    Similarly, assume (for the sake of establishing a contradiction) that $n_4 \in \{1,2\}$ (so $|s_b| \leq 2$). Reducing Equation~\eqref{eq:abc_346_sum} modulo $3$ shows that $s_b$ is an integer multiple of $3$, so $s_b \notin \{-2,-1,1,2\}$, which forces $s_b = 0$. Plugging this back into Equation~\eqref{eq:abc_346_sum} shows that $3s_a + 6s_c = 0$ as well, so the rows of $A$ are contained in the $((n_4 - 1) + (n_3 + n_6 - 1))$-dimensional subspace
    \[
        \big\{ (\mathbf{a}, \mathbf{b}, \mathbf{c}) : \mathbf{a} \in \R^{n_3}, \mathbf{b} \in \R^{n_4}, \mathbf{c} \in \R^{n_6}, \mathbf{b} \cdot \mathbf{1}_{n_4} = \mathbf{a} \cdot \mathbf{1}_{n_3} + \mathbf{c} \cdot \mathbf{1}_{n_6} = 0 \big\}
    \]
    of $\R^{n_3 + n_4 + n_6}$. It follows that $A$ has rank at most $((n_4 - 1) + (n_3 + n_6 - 1)) = n_3 + n_4 + n_6 - 2$, so (by the rank--nullity theorem) $\dim(\mathrm{null}(A)) \geq 2$, which contradicts the fact that $\vv$ is balanced. This contradiction shows that $n_4 \geq 3$, which completes the proof.
\end{proof}

\section{Prime order Laplacian integral and $\{-1,0,1\}$-diagonalizable graphs}\label{sec:prime}

It is straightforward to show that if $G_1$, $G_2$, $\ldots$, $G_d$ are connected and we define
\begin{align}\label{eq:build_larger_G}
    G := (G_1 \sqcup G_2 \sqcup \cdots \sqcup G_d)^\textup{c}
\end{align}
then $G$ is connected too. Furthermore, $G$ is Laplacian integral if and only if $G_1$, $G_2$, $\ldots$, $G_d$ are Laplacian integral. Similarly, recall from \cite[Corollary~2]{johnston2025laplacian} that if $G_1$, $G_2$, $\ldots$, $G_d$ are connected $\{-1,0,1\}$-diagonalizable graphs on $v_1$, $v_2$, $\ldots$, $v_d$ vertices, respectively, then the graph $G$ defined by Equation~\eqref{eq:build_larger_G} is $\{-1,0,1\}$-diagonalizable if and only if $\vv := (v_1, v_2, \ldots, v_d)$ is balanced.

Our main result of this section shows that this method of constructing larger connected Laplacian integral or $\{-1,0,1\}$-diagonalizable graphs from smaller ones gives \emph{all} such graphs of prime order:

\begin{theorem}\label{thm:prime_main}
    Let $n$ be prime and let $G$ be a connected $n$-vertex graph. Then $G$ is Laplacian integral if and only if there exist connected Laplacian integral graphs $G_1$, $G_2$, $\ldots$, $G_d$ ($d \geq 2$) such that
    \begin{align}\label{eq:build_larger_G_2}
        G = (G_1 \sqcup G_2 \sqcup \cdots \sqcup G_d)^\textup{c}.
    \end{align}
    Furthermore, $G$ is $\{-1,0,1\}$-diagonalizable if and only if $G_1$, $G_2$, $\ldots$, $G_d$ are $\{-1,0,1\}$-diagonalizable and $(|V(G_1)|, |V(G_2)|, \ldots, |V(G_d)|)$ is balanced.
\end{theorem}

For example, it was shown in \cite[Table~3]{johnston2025laplacian} that the only connected $\{-1,0,1\}$-diagonalizable graphs of order $7$ are $K_7$, $K_{2,1,1,1,1,1}$, $K_{2,2,1,1,1}$, $K_{3,1,1,1,1}$, and $K_{3,2,1,1}$. Theorem~\ref{thm:prime_main} makes this fact straightforward to see with very little computation: the only balanced vectors $\vv \in \Zp^d$ ($d \geq 2$) with sum equal to $7$, up to permutation of their entries, are $(1,1,1,1,1,1,1)$, $(2,1,1,1,1,1)$, $(2,2,1,1,1)$, $(3,1,1,1,1)$, and $(3,2,1,1)$ (see Table~\ref{tab:balanced_smallsum}), and the only $n$-vertex connected $\{-1,0,1\}$-diagonalizable graphs for $n \leq 3$ are $K_n$, so the only connected $7$-vertex $\{-1,0,1\}$-diagonalizable graphs must be
\begin{align*}
    K_7 & = (K_1 \sqcup K_1 \sqcup K_1 \sqcup K_1 \sqcup K_1 \sqcup K_1 \sqcup K_1)^\textup{c}, \\
    K_{2,1,1,1,1,1} & = (K_2 \sqcup K_1 \sqcup K_1 \sqcup K_1 \sqcup K_1 \sqcup K_1)^\textup{c}, \\
    K_{2,2,1,1,1} & = (K_2 \sqcup K_2 \sqcup K_1 \sqcup K_1 \sqcup K_1)^\textup{c}, \\
    K_{3,1,1,1,1} & = (K_3 \sqcup K_1 \sqcup K_1 \sqcup K_1 \sqcup K_1)^\textup{c}, \ \text{and} \\
    K_{3,2,1,1} & = (K_3 \sqcup K_2 \sqcup K_1 \sqcup K_1)^\textup{c}.
\end{align*}
However, for non-prime orders, there are connected Laplacian integral and $\{-1,0,1\}$-diagonalizable graphs that are not of this form, such as $C_6$ and $K_3 \mathbin{\square} K_3$.

Before proving Theorem~\ref{thm:prime_main}, we require the following lemma:

\begin{lemma}\label{lem:int_eig_implies_Gc_disc}
    Let $n$ be prime and let $G$ be a connected $n$-vertex graph. If $G$ is Laplacian integral, then $G^\textup{c}$ is disconnected.
\end{lemma}

\begin{proof}
    Let the eigenvalues of $L(G)$ be $0 = \lambda_1 < \lambda_2 \leq \cdots \leq \lambda_n$, where the strict inequality follows from connectivity of $G$. By the matrix--tree theorem, if $\tau(G)$ denotes the number of spanning trees of $G$, then
    \[
        \prod_{j=2}^n \lambda_j = n\tau(G).
    \]
    In particular, this implies that at least one of $\lambda_2, \lambda_3, \ldots, \lambda_n$ must be divisible by $n$. Since $0 < \lambda_j \leq n$ for each $j \geq 2$, we conclude that $\lambda_n = n$.
    
    Recall that $L(G^\textup{c}) = nI - J - L(G)$. Since $G$ is connected, every eigenvector of $L(G)$ corresponding to a non-zero eigenvalue is orthogonal to $\mathbf{1}$. If $\mathbf{v}$ is an eigenvector corresponding to $\lambda_n = n$, then we see that
    \[
        L(G^\textup{c})\mathbf{v} = (nI - J - L(G))\mathbf{v} = n\mathbf{v} - \mathbf{0} - n\mathbf{v} = \mathbf{0}.
    \]
    It follows that the eigenspace of $L(G^\textup{c})$ corresponding to the eigenvalue $0$ is at least two dimensional, since it contains the orthogonal vectors $\mathbf{1}$ and $\mathbf{v}$. This implies that $G^\textup{c}$ is disconnected and completes the proof.
\end{proof}

We note that even though we are not aware of Theorem~\ref{thm:prime_main} appearing previously, the same method of proof of Lemma~\ref{lem:int_eig_implies_Gc_disc} has been used to prove properties of prime-order Laplacian integral graphs before \cite{fallat2005graphs}.

\begin{proof}[Proof of Theorem~\ref{thm:prime_main}]
    We already explained why the ``if'' direction holds at the start of this section (for all $n$, not just prime $n$). For the ``only if'' direction, we note that $G^\textup{c}$ is disconnected by Lemma~\ref{lem:int_eig_implies_Gc_disc}. If we refer to the connected components of $G^\textup{c}$ as $G_1$, $G_2$, $\ldots$, $G_d$ ($d \geq 2$), then $G = (G_1 \sqcup G_2 \sqcup \cdots \sqcup G_d)^\textup{c}$. The fact that each $G_j$ is Laplacian integral follows from the eigenvalues of $L(G)$ being $0$, $n$, and the numbers of the form $n - \lambda$, where $\lambda$ is a non-zero eigenvalue of one of the $G_j$'s.
    
    For the ``furthermore'' remark about $\{-1,0,1\}$-diagonalizable graphs, recall from \cite[Proposition~1]{johnston2025laplacian} that if $G$ is $\{-1,0,1\}$-diagonalizable, then so is $G^\textup{c} = G_1 \sqcup G_2 \sqcup \cdots \sqcup G_d$. Note that $L(G^\textup{c})$ is block diagonal, with $L(G_1)$, $L(G_2)$, $\ldots$, $L(G_d)$ as its diagonal blocks. If $P$ is an invertible matrix with all entries from $\{-1,0,1\}$ and $P^{-1} L(G^\textup{c}) P$ is diagonal, then (for each $1 \leq j \leq d$), taking the $j$-th diagonal block of $P$ gives a diagonalization that shows that $G_j$ is also $\{-1,0,1\}$-diagonalizable. The fact that $(v_1, v_2, \ldots, v_d)$ must be balanced follows from the ``only if'' direction of \cite[Corollary~2]{johnston2025laplacian}, which completes the proof.
\end{proof}

In the Laplacian integral case, Theorem~\ref{thm:prime_main} shows that complementation is a bijection between the sets of connected $n$-vertex Laplacian integral graphs and disconnected $n$-vertex Laplacian integral graphs (as long as $n$ is prime). This immediately proves the following result:

\begin{corollary}\label{cor:prime_count_lap_int_half}
    If $n$ is prime then exactly half of all $n$-vertex Laplacian integral graphs are connected.
\end{corollary}

The corresponding statement for $\{-1,0,1\}$-diagonalizable graphs is not true. For example, $C_3 \sqcup C_4$ is $\{-1,0,1\}$-diagonalizable, but its complement is not (see \cite[Theorem~4]{johnston2025laplacian}).

\subsection{Enumerating Laplacian integral graphs}\label{sec:counting_lap}

Theorem~\ref{thm:prime_main} can be used to catalogue and count all Laplacian integral or $\{-1,0,1\}$-diagonalizable graphs of prime order, as long as all such graphs of smaller order are known. We first illustrate this procedure for connected Laplacian integral graphs, before presenting the general counting result.

\begin{example}\label{ex:small_counting}
    To count the number of connected Laplacian integral graphs on $n = 7$ vertices once we have the number of such graphs on fewer vertices, we note from Theorem~\ref{thm:prime_main} that (since $7$ is prime) the connected Laplacian integral graphs on $7$ vertices are exactly the graphs $G$ of the form
    \[
        G = (G_1 \sqcup G_2 \sqcup \cdots \sqcup G_d)^\textup{c},
    \]
    where $G_1$, $G_2$, $\ldots$, $G_d$ are connected Laplacian integral graphs on $6$ or fewer vertices, and $|V(G_1)| + |V(G_2)| + \cdots + |V(G_d)| = 7$. This final condition means exactly that the multiset
    \[
        \big\{|V(G_1)|, |V(G_2)|, \ldots, |V(G_d)|\big\}
    \]
    is an integer partition of $7$.
    
    To catalogue all connected Laplacian integral graphs on $7$ vertices, it thus suffices to (a) compute all integer partitions of $7$ with more than one part (i.e., with all parts being $6$ or less), and (b) associate each partition $p = \{p_1, p_2, \ldots, p_d\}$ with the set of connected Laplacian integral graphs of the form $(G_1 \sqcup G_2 \sqcup \cdots \sqcup G_d)^\textup{c}$, where $|V(G_j)| = p_j$ for all $j \in \{1,2,\ldots,d\}$.
    
    For step (a), we note that there are $14$ integer partitions of $7$ with more than one part \cite{oeisA000065}:
    \begin{align*}
        \{6,1\}, & & \{4,3\}, & & \{3,3,1\}, & & \{3,1,1,1,1\}, & & \{2,1,1,1,1,1\}, \\
        \{5,2\}, & & \{4,2,1\}, & & \{3,2,2\}, & & \{2,2,2,1\}, & & \{1,1,1,1,1,1,1\}, \\
        \{5,1,1\}, & & \{4,1,1,1\}, & & \{3,2,1,1\}, & & \{2,2,1,1,1\}. & &
    \end{align*}

    For step (b), we recall \cite{oeisA363064} that there are $1$, $1$, $2$, $5$, $12$, and $37$ connected Laplacian integral graphs on $1$, $2$, $3$, $4$, $5$, and $6$ vertices, respectively. This tells us the following:
    \begin{itemize}
        \item The partition $\{6,1\}$ gives $37$ connected Laplacian integral graphs on $7$ vertices: the graphs of the form $(G_1 \sqcup G_2)^\textup{c}$, where $G_1$ is one of the $37$ $6$-vertex graphs and $G_2$ is the unique $1$-vertex graph.
        
        \item The partition $\{4,3\}$ gives $5 \cdot 2 = 10$ connected Laplacian integral graphs on $7$ vertices: the graphs of the form $(G_1 \sqcup G_2)^\textup{c}$, where $G_1$ is one of the $5$ $4$-vertex graphs and $G_2$ is one of the $2$ $3$-vertex graphs.
        
        \item The other integer partitions are handled similarly, with one wrinkle if there are multiple parts of the partition of the same size. For the partition $\{3,3,1\}$, for example, the associated graphs all have the form $(G_1 \sqcup G_2 \sqcup G_3)^\textup{c}$, where $G_3 = K_1$ is the unique connected Laplacian integral graph on $1$ vertex, and $G_1,G_2 \in \{K_3,K_{2,1}\}$ are each one of the two connected Laplacian integral graphs on $3$ vertices. However, there are only three $7$-vertex graphs of this form, not $2 \cdot 2 = 4$ as we might first expect, since $(K_3 \sqcup K_{2,1} \sqcup K_1)^\textup{c}$ and $(K_{2,1} \sqcup K_3 \sqcup K_1)^\textup{c}$ are isomorphic.
    \end{itemize}
    Once we sum these counts over each of the $14$ integer partitions of $7$, we see that there must be
    \[
        37 + 10 + 3 + 2 + 1 + 12 + 5 + 2 + 1 + 1 + 12 + 5 + 2 + 1 = 94
    \]
    connected Laplacian integral graphs on $7$ vertices (which matches the known count from \cite{oeisA363064}).
\end{example}

This method of enumerating graphs works whenever the graphs of interest have the form $(G_1 \sqcup G_2 \sqcup \cdots \sqcup G_d)^\textup{c}$ (or, since complementation doesn't change the count, $G_1 \sqcup G_2 \sqcup \cdots \sqcup G_d$), as long as we know all valid multisets $\{|V(G_1)|, |V(G_2)|, \ldots, |V(G_d)|\}$ and all valid $G_j$'s. Recall that $P_n$ denotes the set of integer partitions of $n$ and $m_p(k)$ denotes the multiplicity of $k$ in the multiset $p \in P_n$ (i.e., the number of occurrences of $k$ in $p$). Then we have the following result:

\begin{theorem}\label{thm:counting_graphs}
    Let $n$ be a positive integer, let $S_n \subseteq P_n$, and let $\mathcal{C}_1$, $\mathcal{C}_2$, $\mathcal{C}_3$, $\ldots$ be sets of connected graphs on $1$, $2$, $3$, $\ldots$ vertices, respectively. Let $\mathcal{G}_n$ be the set of $n$-vertex graphs $G$ that have the following properties, where we denote the connected components of $G$ by $G_1$, $G_2$, $\ldots$, $G_d$:
    \begin{itemize}
        \item $G_k \in \mathcal{C}_{|V(G_k)|}$ for all $k \in \{1,2,\ldots,d\}$, and 
        \item $\{|V(G_1)|, |V(G_2)|, \ldots, |V(G_d)|\} \in S_n$.
    \end{itemize}
    Then
    \begin{align}\label{eq:Gn_size_formula}
        |\mathcal{G}_n| = \sum_{p \in S_n} \prod_{k = 1}^n \binom{|\mathcal{C}_k|+m_{p}(k)-1}{m_{p}(k)}.
    \end{align}
\end{theorem}

\begin{proof}
    Since each $p = \{p_1,p_2,\ldots,p_{|p|}\} \in S_n$ can be used to construct the graphs $G = G_1 \sqcup G_2 \sqcup \cdots \sqcup G_{|p|} \in \mathcal{G}_n$, where $G_k \in \mathcal{C}_{p_k}$ for all $k \in \{1,2,\ldots,|p|\}$, it suffices to count how many such graphs there are for each $p$ and then sum over $p \in S_n$.

    Fix $k \in \{1,2,\ldots,n\}$ and let $j_1$, $j_2$, $\ldots$, $j_{m_p(k)}$ be the indices $j$ for which $p_j = k$. The number of ways of choosing $G_{j_1}, G_{j_2}, \ldots, G_{j_{m_p(k)}} \in \mathcal{C}_k$, up to isomorphism of $G_{j_1} \sqcup G_{j_2} \sqcup \cdots \sqcup G_{j_{m_p(k)}}$, is
    \begin{align}\label{eq:binom_formula}
        \binom{|\mathcal{C}_k|+m_{p}(k)-1}{m_{p}(k)}.
    \end{align}
    To see this, notice that it is equal to the number of tuples $(\ell_1,\ell_2,\ldots,\ell_{m_p(k)}) \in \{1,2,\ldots,|\mathcal{C}_k|\}^{m_p(k)}$ (with $\ell_i$ indexing which member of $\mathcal{C}_k$ is used for $G_{j_i}$) that satisfy $\ell_1 \leq \ell_2 \leq \cdots \leq \ell_{m_p(k)}$; the binomial coefficient~\eqref{eq:binom_formula} is well-known to count such tuples (e.g., this quantity is denoted by $\big(\binom{|\mathcal{C}_k|}{m_p(k)}\big)$ in \cite[Section~1.2]{stanley2011enumerative}).

    For a given $p \in S_n$, it follows that multiplying the binomial coefficient~\eqref{eq:binom_formula} for $k = 1, 2, \ldots n$ counts the number of graphs (up to isomorphism) of the form $G = G_1 \sqcup G_2 \sqcup \cdots \sqcup G_{|p|}$ for which $|V(G_j)| = p_j$ for all $j \in \{1,2,\ldots,|p|\}$. Summing over all $p \in S_n$ gives the result.
\end{proof}

We note that if $|\mathcal{C}_k| = 0$, then the binomial coefficient in Equation~\eqref{eq:Gn_size_formula} should follow the standard conventions that $\binom{a-1}{a} = 0$ when $a \geq 1$ and $\binom{-1}{0} = 1$. We also note that if $S_n = P_n$, then Theorem~\ref{thm:counting_graphs} reduces to the standard Euler transform for computing the number of not necessarily connected graphs with some property from the number of such connected graphs (see \cite[page~90]{HP73} and Corollary~\ref{cor:prime_count_lap_int_b}, for example). However, most of the cases of interest to us arise when $S_n \subsetneq P_n$.

Theorem~\ref{thm:counting_graphs} can be used to count many families of graphs of interest to us. For example, if $n$ is prime and $\mathcal{G}_n$ is the set of complements of connected $n$-vertex Laplacian integral graphs, then Theorem~\ref{thm:prime_main} tells us that $G \in \mathcal{G}_n$ if and only if the connected components of $G$ (i.e., $G_1$, $G_2$, $\ldots$, $G_d$ for $d \geq 2$) are each Laplacian integral. In other words, Theorem~\ref{thm:counting_graphs} applies with $\mathcal{C}_k$ being the set of $k$-vertex connected Laplacian integral graphs and $S_n$ being the set of integer partitions of $n$ with more than one part (since $d \geq 2$): $S_n = P_n\setminus\{n\}$. Even more explicitly, we get the following corollary:

\begin{corollary}\label{cor:prime_count_lap_int}
    Let $cl(k)$ denote the number of connected $k$-vertex Laplacian integral graphs. If $n$ is prime then
    \begin{align}\label{eq:cl_recurrence}
        cl(n) = \sum_{p \in P_n\setminus\{n\}} \prod_{k=1}^{n-1} \binom{cl(k)+m_p(k)-1}{m_p(k)}.
    \end{align}
\end{corollary}

For example, if $n = 7$ then the sum described by Corollary~\ref{cor:prime_count_lap_int} is exactly the sum of $14$ terms at the end of Example~\ref{ex:small_counting}. More importantly, when we combine this method with our explicit computation of all connected Laplacian integral graphs of order up to $12$, Corollary~\ref{cor:prime_count_lap_int} in the $n = 13$ case tells us that there are $91,918$ connected Laplacian integral graphs of order $13$; see Table~\ref{tab:graph_counts_thirteen}.

By instead letting $\mathcal{G}_n$ be the set of (complements of) all (not necessarily connected) Laplacian integral graphs, and letting $S_n = P_n$, we get a formula for counting the number of (not necessarily connected) Laplacian integral graphs in terms of the number of connected Laplacian integral graphs, even when $n$ is not prime:

\begin{corollary}\label{cor:prime_count_lap_int_b}
    Let $l(k)$ and $cl(k)$ denote the number of $k$-vertex Laplacian integral and connected $k$-vertex Laplacian integral graphs, respectively. Then
    \[
        l(n) = \sum_{p \in P_n} \prod_{k=1}^{n} \binom{cl(k)+m_p(k)-1}{m_p(k)}.
    \]
\end{corollary}

For example, if $n = 12$ then Corollary~\ref{cor:prime_count_lap_int_b} tells us that there are $62,052$ Laplacian integral graphs of order $12$ (it also tells us that there are $183,836$ Laplacian integral graphs of order $13$, but that number is more easily obtained from Corollary~\ref{cor:prime_count_lap_int_half} since $13$ is prime); see Table~\ref{tab:graph_counts_thirteen}.

\subsection{Enumerating $\{-1,0,1\}$-diagonalizable graphs}\label{sec:counting_zero_one}

The process for using Theorems~\ref{thm:prime_main} and~\ref{thm:counting_graphs} to enumerate $\{-1,0,1\}$-diagonalizable graphs is much the same as it was for Laplacian integral graphs. The only substantial difference is that, instead of summing over the set of non-singleton integer partitions $P_n \setminus \{n\}$ when computing the number of connected graphs, we sum over the set of balanced vectors with positive integer entries that sum to $n$ (re-interpreted as a multiset instead of a vector in the natural way). We denote the set of all such ``balanced multisets'' by $B_n$, and we leave the details of how to enumerate these multisets to Appendix~\ref{app:balanced}.

Because of this restriction to balanced vectors, and the fact that the largest entry in a balanced vector $\vv \in \Zp^d$ ($d \geq 2$) that sums to $n$ cannot exceed $\lfloor n/2 \rfloor$ (Lemma~\ref{lem:balanced_max_entry}), the product formula for counting connected $\{-1,0,1\}$-diagonalizable graphs when $n$ is prime can be truncated at $\lfloor n/2 \rfloor$ rather than at $n-1$:

\begin{corollary}\label{cor:prime_count_nozo_int}
    Let $cs(k)$ denote the number of connected $k$-vertex $\{-1,0,1\}$-diagonalizable graphs. If $n$ is prime then
    \begin{align}\label{eq:cu_recurrence}
        cs(n) = \sum_{p \in B_n \setminus \{n\}} \prod_{k=1}^{\lfloor n/2 \rfloor} \binom{cs(k)+m_p(k)-1}{m_p(k)}.
    \end{align}
\end{corollary}

In particular, if we have a complete list of $\{-1,0,1\}$-diagonalizable graphs up to a given order, then this method lets us enumerate and count all such graphs of prime order roughly twice as large. Since our computational techniques produced all connected $\{-1,0,1\}$-diagonalizable graphs on up to $12$ vertices (see Appendix~\ref{sec:appendix_compute_lapint}), this method produces all such graphs on $13$, $17$, $19$, and $23$ vertices (see Table~\ref{tab:prime_101_all}).

\begin{table}[!htb]
    \centering
    \begin{tabular}{clc}\toprule
        \# of vertices & connected $\{-1,0,1\}$-diagonalizable graphs & \# of graphs\\\toprule
        $2$ & $K_2$ & $1$ \\\midrule
        $3$ & $K_3$ & $1$ \\\midrule
        $5$ & $K_5$, \ $K_{2,1,1,1}$ & $2$ \\\midrule
        $7$ & $K_7$, \ $K_{3,2,1,1}$, \ $K_{3,1,1,1,1}$, \ $K_{2,2,1,1,1}$, \ $K_{2,1,1,1,1,1}$ & $5$ \\\midrule
        $11$ & $K_{11}$, \ $K_{5,3,1,1,1}$, \ $K_{5,2,2,1,1}$, \ $K_{5,2,1,1,1,1}$, \ $K_{5,1,1,1,1,1,1}$, \ $K_{4,3,2,1,1}$, \ $K_{4,3,1,1,1,1}$, & $35$ \\
        & $K_{4,2,2,1,1,1}$, \ $K_{4,2,1,1,1,1,1}$, \ $K_{4,1,1,1,1,1,1,1}$, \ $K_{3,3,2,2,1}$, \ $K_{3,3,2,1,1,1}$, \ $K_{3,3,1,1,1,1,1}$, & \\
        & $K_{3,2,2,2,1,1}$, \ $K_{3,2,2,1,1,1,1}$, \ $K_{3,2,1,1,1,1,1,1}$, \ $K_{3,1,1,1,1,1,1,1,1}$, \ $K_{2,2,2,2,1,1,1}$, & \\
        & $K_{2,2,2,1,1,1,1,1}$, \ $K_{2,2,1,1,1,1,1,1,1}$, \ $K_{2,1,1,1,1,1,1,1,1,1}$, & \\
        & $(K_{2,1,1,1} \sqcup K_3 \sqcup 3K_1)^\textup{c}$, \ $(K_{2,1,1,1} \sqcup 2K_2 \sqcup 2K_1)^\textup{c}$, \ $(K_{2,1,1,1} \sqcup K_2 \sqcup 4K_1)^\textup{c}$, & \\
        & $(K_{2,1,1,1} \sqcup 6K_1)^\textup{c}$, \ $(K_{2,2} \sqcup K_3 \sqcup K_2 \sqcup 2K_1)^\textup{c}$, \ $(K_{2,1,1} \sqcup K_3 \sqcup K_2 \sqcup 2K_1)^\textup{c}$, & \\
        & $(K_{2,2} \sqcup K_3 \sqcup 4K_1)^\textup{c}$, \ $(K_{2,1,1} \sqcup K_3 \sqcup 4K_1)^\textup{c}$, \ $(K_{2,2} \sqcup 2K_2 \sqcup 3K_1)^\textup{c}$, & \\
        & $(K_{2,1,1} \sqcup 2K_2 \sqcup 3K_1)^\textup{c}$, \ $(K_{2,2} \sqcup K_2 \sqcup 5K_1)^\textup{c}$, \ $(K_{2,1,1} \sqcup K_2 \sqcup 5K_1)^\textup{c}$, & \\
        & $(K_{2,2} \sqcup 7K_1)^\textup{c}$, \ $(K_{2,1,1} \sqcup 7K_1)^\textup{c}$ & \\\midrule
        $13$ & Too many to list; see \cite{VaronaCode}. & $124$ \\\midrule
        $17$ & -- & $1,155$ \\\midrule
        $19$ & -- & $3,449$ \\\midrule
        $23$ & -- & $33,664$ \\\bottomrule
    \end{tabular}
    \caption{All connected $\{-1,0,1\}$-diagonalizable graphs of small prime orders. See Table~\ref{tab:graph_counts_thirteen} for known non-prime counts.}\label{tab:prime_101_all}
\end{table}

Since a graph $G$ is $\{-1,0,1\}$-diagonalizable if and only if the same is true of each of its connected components \cite[Proposition~1(b)]{johnston2025laplacian}, we also get a method of turning counts of connected $\{-1,0,1\}$-diagonalizable graphs into counts of all (not necessarily connected) such graphs:

\begin{corollary}\label{cor:prime_count_nozo_int_b}
    Let $s(k)$ and $cs(k)$ denote the number of $k$-vertex $\{-1,0,1\}$-diagonalizable and connected $k$-vertex $\{-1,0,1\}$-diagonalizable graphs, respectively. Then
    \[
        s(n) = \sum_{p \in P_n} \prod_{k=1}^{n} \binom{cs(k)+m_p(k)-1}{m_p(k)}.
    \]
\end{corollary}

For example, if $n = 13$ then Corollary~\ref{cor:prime_count_nozo_int_b} tells us that there are $1,445$ $\{-1,0,1\}$-diagonalizable graphs of order $13$ (see Table~\ref{tab:graph_counts_thirteen} for counts related to other values of $n$).

\section{Regular integral graphs}\label{sec:prime_regular}

We now specialize even further and enumerate \emph{regular} integral and $\{-1,0,1\}$-diagonalizable graphs (for regular graphs, Laplacian integrality is the same as adjacency integrality, so we can safely drop the ``Laplacian'' qualifier). Define the following four functions that count regular graphs of various types:
\begin{itemize}
    \item $rl(n)$: the number of regular integral graphs of order $n$.
    
    \item $crl(n)$: the number of connected regular integral graphs of order $n$.
    
    \item $rs(n)$: the number of regular $\{-1,0,1\}$-diagonalizable graphs of order $n$.
    
    \item $crs(n)$: the number of connected regular $\{-1,0,1\}$-diagonalizable graphs of order $n$.
\end{itemize}
We also use $rl(n,r)$ to denote the number of $r$-regular integral graphs of order $n$, and similarly for $crl(n,r)$, $rs(n,r)$, and $crs(n,r)$.

We used our computational toolkit to determine all regular graphs of these types up to order~$15$ (see Table~\ref{tab:graph_counts_thirteen} and \cite{VaronaCode}), and we can use Theorem~\ref{thm:prime_main} to extend this enumeration to larger prime orders. In fact, if $n$ is prime, then Theorem~\ref{thm:prime_main} becomes more useful than it might initially appear to be. Because each of $G_1$, $G_2$, $\ldots$, $G_d$ must be regular \emph{of the same degree} in order for $G$ to be regular, the form of $G$ is heavily restricted. For example, since $d \geq 2$, each $G_j$ must have degree at most $\lfloor n/2 \rfloor$ (if this were not the case, then $G$ would have strictly more than $d \cdot (\lfloor n/2 \rfloor + 1) \geq n$ vertices). As a result, the recurrence described by Theorem~\ref{thm:counting_graphs}, in the case of connected regular integral graphs, depends only on the connected regular integral graphs of order up to $\lfloor n/2 \rfloor$, not up to $n-1$. A more careful version of this analysis leads to the following result:

\begin{corollary}\label{cor:prime_reg_count_lap_int}
    Let $0 \leq r < n$ be integers and let $P_{n,r}$ denote the set of integer partitions of $n$ for which each part is between $r+1$ and $n-r-1$ inclusive. If $n$ is prime then
    \begin{align}\label{eq:crl_recurrence}
        crl(n,n-r-1) & = \begin{cases}
            0 & \text{if $r$ is odd or $r \geq 2\left\lfloor \frac{n+1}{4} \right\rfloor$},\\
            \displaystyle\sum_{p \in P_{n,r}} \prod_{k=r+1}^{n-r-1} \binom{crl(k,r)+m_p(k)-1}{m_p(k)} & \text{otherwise, and}
        \end{cases} \\\label{eq:crl_recurrenceb}
        rl(n,r) & = rl(n,n-r-1) = crl(n,n-r-1) \quad \text{for} \quad r \leq \frac{n-1}{2}.
    \end{align}
    Furthermore, if $r \leq (n-1)/2$ and $n$ is not necessarily prime then
    \begin{align}\label{eq:crl_nonprime}
        crl(n,n-r-1) & = crl(n,r) + \sum_{p \in P_{n,r}} \prod_{k=r+1}^{n-r-1} \binom{crl(k,r)+m_p(k)-1}{m_p(k)}.
    \end{align}
\end{corollary}

\begin{proof}
    These claims are all trivial when $n \leq 2$, so throughout the proof we assume that $n \geq 3$.
    
    If $n$ is prime (and thus odd) then the claim that $crl(n,n-r-1) = 0$ when $r$ is odd follows from the handshaking lemma. The claim that $crl(n,n-r-1) = 0$ when $r \geq 2\left\lfloor \frac{n+1}{4} \right\rfloor$ will follow from the general formula~\eqref{eq:crl_recurrence} (once we have established it), since $n-r-1 < r+1$ when $n$ is odd and $r \geq 2\left\lfloor \frac{n+1}{4} \right\rfloor$.

    The general formula~\eqref{eq:crl_recurrence} follows from Theorem~\ref{thm:counting_graphs} when $\mathcal{G}_n$ is the set of connected $n$-vertex $(n-r-1)$-regular graphs and $\mathcal{C}_n$ is the set of connected $n$-vertex $r$-regular graphs. The bounds on $k$ come from the fact that every $r$-regular graph $G_j$ must have at least $r+1$ vertices, and to appear in a decomposition of the form $G = (G_1 \sqcup G_2 \sqcup \cdots \sqcup G_d)^{\textup{c}}$ with $d \geq 2$ must have at most $n-r-1$ vertices (so that at least one other $G_j$ with at least $r+1$ vertices can be present).

    The equality $rl(n,n-r-1) = crl(n,n-r-1)$ of Equation~\eqref{eq:crl_recurrenceb} comes from the fact that if $r \leq (n-1)/2$ then all $(n-r-1)$-regular graphs on $n$ vertices are connected (since there aren't enough vertices for two $(n-r-1)$-regular connected components). The equality $rl(n,r) = rl(n,n-r-1)$ comes from noticing that complementation is a bijection between the sets of $r$-regular and $(n-r-1)$-regular graphs on $n$ vertices.

    All that remains is to prove Equation~\eqref{eq:crl_nonprime} when $n$ is not necessarily prime. To this end, we simply apply Theorem~\ref{thm:counting_graphs} exactly as we did in the prime case; the only difference here when counting connected $(n-r-1)$-regular graphs $G$ is that we don't have the guarantee (provided by Theorem~\ref{thm:prime_main} when $n$ is prime) that $G^{\textup{c}}$ is disconnected (i.e., it's possible that $d = 1$). The extra term $crl(n,r)$ accounts for this $d = 1$ case.
\end{proof}

\begin{example}\label{exam:n17_regular_integral}
    To illustrate how the counting of Corollary~\ref{cor:prime_reg_count_lap_int} works in the prime case, consider the problem of computing the number of connected regular integral graphs of order $17$. We need to compute $crl(n,n-r-1)$ for $r \in \{0,2,4,6\}$.

    If $r = 0$ then $crl(17,16) = 1$ since the only $16$-regular graph on $17$ vertices is $K_{17}$, and it is connected and integral.

    If $r = 2$ then Equation~\eqref{eq:crl_recurrence} says that
    \[
        crl(17,14) = \sum_{p \in P_{17,2}} \prod_{k=3}^{14} \binom{crl(k,2)+m_p(k)-1}{m_p(k)}.
    \]
    Since the only connected $2$-regular integral graphs are the cycle graphs $C_3$, $C_4$, and $C_6$ \cite{GM94}, we have $crl(k,2) = 1$ if $k \in \{3,4,6\}$ and $crl(k,2) = 0$ otherwise, so the only non-zero contributions to this sum arise when $k \in \{3,4,6\}$. The only integer partitions of $17$ with all parts equal to $3$, $4$, or $6$ are
    \[
        \{3,3,3,4,4\} \quad \text{and} \quad \{3,4,4,6\},
    \]
    so $crl(17,14) = 2$ (the two $17$-vertex connected $14$-regular integral graphs are $(3C_3 \sqcup 2C_4)^{\textup{c}}$ and $(C_3 \sqcup 2C_4 \sqcup C_6)^{\textup{c}}$).
    
    If $r = 4$ then
    \[
        crl(17,12) = \sum_{p \in P_{17,4}} \prod_{k=5}^{12} \binom{crl(k,4)+m_p(k)-1}{m_p(k)}.
    \]
    The only integer partitions of $17$ with all parts between $r+1 = 5$ and $n-r-1 = 12$, inclusive, are
    \[
        \{5,5,7\}, \ \{5,6,6\}, \ \{5, 12\}, \ \{6, 11\}, \ \{7,10\}, \ \text{and} \ \{8,9\}.
    \]
    Since the numbers of connected $4$-regular integral graphs on $5$, $6$, $7$, $8$, $9$, $10$, $11$, and $12$ vertices are $1$, $1$, $1$, $2$, $4$, $1$, $0$, and $8$ respectively (see Table~\ref{tab:graph_counts_regular}), we have
    \begin{align*}
        crl(17,12) = 1 \cdot 1 + 1 \cdot 1 + 1 \cdot 8 + 1 \cdot 0 + 1 \cdot 1 + 2 \cdot 4 = 19.
    \end{align*}

    Finally, if $r = 6$ then
    \[
        crl(17,10) = \sum_{p \in P_{17,6}} \prod_{k=7}^{10} \binom{crl(k,6)+m_p(k)-1}{m_p(k)}.
    \]
    The only integer partitions of $17$ with all parts between $r+1 = 7$ and $n-r-1 = 10$, inclusive, are $\{7,10\}$ and $\{8,9\}$. Since $crl(7,6) = 1$, $crl(8,6) = 1$, $crl(9,6) = 2$, and $crl(10,6) = 5$ (see Table~\ref{tab:graph_counts_regular}), we have
    \begin{align*}
        crl(17,10) = 1 \cdot 5 + 1 \cdot 2 = 7.
    \end{align*}
    It follows that
    \begin{align*}
        crl(17) & = crl(17,16) + crl(17,14) + crl(17,12) + crl(17,10) \\
        & = 1 + 2 + 19 + 7 \\
        & = 29.
    \end{align*}
\end{example}

Similarly, if $n = 17$ then Corollary~\ref{cor:prime_reg_count_lap_int} tells us that there are $1$, $2$, $19$, and $7$ regular integral graphs of degree $0$, $2$, $4$, and $6$, respectively, and they are all disconnected. Despite $r = 8$ being even, there are no $8$-regular integral graphs of order $17$, a fact that we clarify with the following corollary:

\begin{corollary}\label{cor:reg_lap_prime_existence}
    Suppose $n \geq 3$ is prime and $r \in \{0,1,2,\ldots,n-1\}$.
    \begin{itemize}
        \item $rl(n,r) > 0$ if and only if $r$ is even and $r \neq (n-1)/2$.
        \item $crl(n,r) > 0$ if and only if $r$ is even and $r > (n-1)/2$.
    \end{itemize}
\end{corollary}

\begin{proof}
    The ``only if'' directions of both bullet points follow immediately from Corollary~\ref{cor:prime_reg_count_lap_int} (notice that, if $r$ is even, $r = (n-1)/2$ is equivalent to $n \equiv 1 \pmod{4}$ and $r = 2\lfloor (n+1)/4\rfloor$).

    For the ``if'' directions, we prove the following slightly stronger assertion by induction on $n$: for every (not necessarily prime) $n \geq 3$, if $r < (n-1)/2$ and $nr$ is even, then $rl(n,r) > 0$. If we can prove this assertion, then the ``if'' direction of the first bullet point of the corollary will follow since $n \geq 3$ being prime implies $n$ is odd, so $nr$ being even is equivalent to $r$ being even. The second bullet point of the corollary will then follow by complementation.
    
    Notice (e.g., by Tables~\ref{tab:graph_counts_regular} and~\ref{tab:graph_counts_regular_disconnected}) that the assertion is true when $3 \leq n \leq 14$, which establishes the base case of the induction. Now suppose $n \geq 15$ and that the assertion has been proved for all smaller orders. Suppose $r < (n-1)/2$ and $nr$ is even. If $r = 0$, the empty graph on $n$ vertices shows that $rl(n,r) > 0$. If $r = 1$ then $n$ must be even, so $(n/2)K_2$ shows that $rl(n,r) > 0$. We can thus assume from now on that $r \geq 2$.
    
    We claim that if $n \geq 2r+2$ (i.e., $r < (n-1)/2$), then there exist integers $d \geq 2$ and $m_1, m_2, \ldots, m_d \in \{r+1,r+2,\ldots, 2r\}$ such that $n = m_1 + m_2 + \cdots + m_d$ and $rm_j$ is even for all $j \in \{1,2,\ldots,d\}$. If we can prove this claim, then we will be done with the inductive step (and the proof) since, for all $j \in \{1,2,\ldots,d\}$, we can let $G_j$ be a connected $r$-regular integral graph of order $m_j$ (which is guaranteed to exist by the inductive hypothesis), which results in $G_1 \sqcup G_2 \sqcup \cdots \sqcup G_d$ being an $r$-regular integral graph of order $n$. It thus suffices to prove the claim, and to this end we split into two cases.
    
    \textbf{Case 1:} $r$ is even. Then the requirement that $rm_j$ is even for all $j$ is automatically satisfied; we do not need to worry about the parity of $m_j$. The claim is true if $2r+2 \leq n \leq 3r+1$ since the integers between $2r+2$ and $3r+1$ inclusive are
    \[
        (r+1)+(r+1), (r+1)+(r+2), \ldots, (r+1)+2r.
    \]
    The claim is true if $n = 3r+2$ since $3r+2 = (r+2)+2r$. Finally, the claim is true if $n \geq 3r+3$ by adding more terms equal to $r+1$ to one of these 2-term decompositions.
    
    \textbf{Case 2:} $r$ is odd. Then $n$ must be even, and we need to show that we can find even $m_1, m_2, \ldots, m_d \in \{r+1,r+2,\ldots, 2r\}$ such that $n = m_1 + m_2 + \cdots + m_d$. It thus suffices to show that we can write $n/2$ as a sum of integers from $\{(r+1)/2, (r+3)/2, \ldots, r\}$. Since $n/2 \geq r+1$, this follows from the exact same argument as in Case~1, which completes the proof of the claim, and thus the proof of the inductive step, and thus the proof of the corollary.
\end{proof}

In the special case when $n = 13$ and $r = 6$, Corollary~\ref{cor:reg_lap_prime_existence} shows that there does not exist a $6$-regular integral graph of order $13$, thus solving \cite[Problem~3]{BSZ09}. (This problem was also solved by our computational results (see Table~\ref{tab:graph_counts_regular_disconnected}), but Corollary~\ref{cor:reg_lap_prime_existence} better explains ``why'' there are no such graphs in this case.)

When $n$ is not prime, we can use Equation~\eqref{eq:crl_nonprime} in much the same way that we used Equation~\eqref{eq:crl_recurrence} in the prime case in Example~\ref{exam:n17_regular_integral}. For example, if $n = 12$ and $r = 5$ then (because the only integer partitions of $12$ with all parts in $\{r+1,\ldots,n-r-1\}$ is $\{6,6\}$) we learn that
\begin{align}\label{eq:crl_126}
    crl(12,6) = crl(12,5) + crl(6,5) = crl(12,5) + 1,
\end{align}
with the final equality following from the fact that the only $5$-regular $6$-vertex graph is $K_6$. Computation shows that $crl(12,5) = 13$ (see Table~\ref{tab:graph_counts_regular}), from which Equation~\eqref{eq:crl_126} tells us that $crl(12,6) = 14$. The following corollary generalizes this observation.

\begin{corollary}\label{cor:reg_middle_vals_even}
    If $n$ is even then $\displaystyle crl\left(n,\frac{n}{2}\right) = crl\left(n,\frac{n}{2}-1\right) + 1$.
\end{corollary}

\begin{proof}
    Equation~\eqref{eq:crl_nonprime} tells us that if $r = n/2 - 1$ then
    \begin{align}\begin{split}\label{eq:reg_mid_vals_even}
        crl\left(n,\frac{n}{2}\right) & = crl(n,n-r-1) \\
        & = crl(n,r) + \sum_{p \in P_{n,r}} \prod_{k=r+1}^{n-r-1} \binom{crl(k,r)+m_p(k)-1}{m_p(k)} \\
        & = crl\left(n,\frac{n}{2}-1\right) + \sum_{p \in P_{n,n/2-1}} \prod_{k=n/2}^{n/2} \binom{crl\left(k,\tfrac{n}{2}-1\right)+m_p(k)-1}{m_p(k)}.
    \end{split}\end{align}
    The only member of $P_{n,n/2-1}$ is the partition $\{n/2,n/2\}$ and the only connected integral $r$-regular graph on $r+1 = n/2$ vertices is $K_{r+1}$. It follows that
    \[
        \sum_{p \in P_{n,n/2-1}} \prod_{k=n/2}^{n/2} \binom{crl\left(k,\tfrac{n}{2}-1\right)+m_p(k)-1}{m_p(k)} = \binom{crl\left(\tfrac{n}{2},\tfrac{n}{2}-1\right)+1}{2} = 1.
    \]
    Combining this fact with Equation~\eqref{eq:reg_mid_vals_even} gives the result.
\end{proof}

A very similar argument gives a similar result when $n \equiv 3 \pmod{4}$:

\begin{corollary}\label{cor:reg_middle_vals}
    If $n \equiv 3 \pmod{4}$ then $\displaystyle crl\left(n,\frac{n+1}{2}\right) = crl\left(n,\frac{n-3}{2}\right) + 1$.
\end{corollary}

\begin{proof}
    Equation~\eqref{eq:crl_nonprime} tells us that if $r = (n-3)/2$ then
    \begin{align}\begin{split}\label{eq:reg_mid_vals}
        crl\left(n,\frac{n+1}{2}\right) & = crl(n,n-r-1) \\
        & = crl(n,r) + \sum_{p \in P_{n,r}} \prod_{k=r+1}^{n-r-1} \binom{crl(k,r)+m_p(k)-1}{m_p(k)} \\
        & = crl\left(n,\frac{n-3}{2}\right) + \sum_{p \in P_{n,(n-3)/2}} \prod_{k=(n-1)/2}^{(n+1)/2} \binom{crl\left(k,\tfrac{n-3}{2}\right)+m_p(k)-1}{m_p(k)}.
    \end{split}\end{align}
    The only member of $P_{n,(n-3)/2}$ is the partition $\{(n-1)/2,(n+1)/2\}$, the only connected integral $r$-regular graph on $r+1 = (n-1)/2$ vertices is $K_{r+1}$, and the only connected integral $r$-regular graph on $r+2 = (n+1)/2$ vertices is $K_{2,2,\ldots,2}$. It follows that
    \[
        \sum_{p \in P_{n,(n-3)/2}} \prod_{k=(n-1)/2}^{(n+1)/2} \binom{crl\left(k,\tfrac{n-3}{2}\right)+m_p(k)-1}{m_p(k)} = \binom{crl\left(\tfrac{n-1}{2},\tfrac{n-3}{2}\right)}{1}\binom{crl\left(\tfrac{n+1}{2},\tfrac{n-3}{2}\right)}{1} = 1.
    \]
    Combining this fact with Equation~\eqref{eq:reg_mid_vals} gives the result.
\end{proof}

\begin{corollary}\label{cor:reg_middle_vals_prime}
    If $n \equiv 3 \pmod{4}$ is prime and $r = (n+1)/2$ then $crl(n,r) = 1$, and the unique connected integral $r$-regular order-$n$ graph is\[\big(K_{(n-1)/2} \sqcup K_{2,2,\ldots,2}\big)^{\textup{c}},\]where there are $(n+1)/4$ parts in the complete multipartite graph $K_{2,2,\ldots,2}$.
\end{corollary}

\begin{proof}
    This follows from Corollary~\ref{cor:reg_middle_vals} via the fact that $crl(n,(n-3)/2) = 0$ when $n$ is prime (due to Equation~\eqref{eq:crl_recurrence} of Corollary~\ref{cor:prime_reg_count_lap_int}).
\end{proof}

As noted a few times already, $rl(n,0) = crl(n,n-1) = 1$ always holds, even if $n$ is not prime (the single $0$-regular integral graph on $n$ vertices is the empty graph, and the single connected $(n-1)$-regular graph on $n$ vertices is the complete graph). Similarly, $rl(n,1) = crl(n,n-2) = 1$ if $n$ is even and $rl(n,1) = crl(n,n-2) = 0$ if $n$ is odd, even if $n$ is not prime (the latter case follows from the handshaking lemma and the former case follows from the fact that $K_{2,2,\ldots,2}$ is the only $(n-2)$-regular $n$-vertex graph, and it is Laplacian integral).\footnote{We used this observation in the proof of Corollary~\ref{cor:reg_middle_vals}.} The following corollary extends this line of reasoning to the next level.

\begin{corollary}\label{cor:prime_reg_count_lap_int_2reg}
    Suppose $n \geq 5$ is an integer. Then
    \begin{align*}
        rl(n,2) = crl(n,n-3) & = \left(\left\lfloor \frac{n+3}{3}\right\rfloor - \left\lfloor \frac{n+1}{4}\right\rfloor\right)\left(\left\lfloor \frac{n+3}{3}\right\rfloor + \left\lfloor \frac{n+1}{4}\right\rfloor - \left\lfloor \frac{n+1}{2}\right\rfloor\right).
    \end{align*}
    Furthermore, the $2$-regular integral graphs and the connected $(n-3)$-regular integral graphs are exactly those of the form
    \[
        aC_3 \sqcup bC_4 \sqcup cC_6 \quad \text{and} \quad (aC_3 \sqcup bC_4 \sqcup cC_6)^{\textup{c}},
    \]
    respectively, where $(a,b,c) \in \Znn^3$ are such that $3a + 4b + 6c = n$.
\end{corollary}

\begin{proof}
    The formula is known to count the number of non-negative integer tuples $(a,b,c)$ for which $3a+4b+6c = n$ (see \cite{oeisA025828}), so it suffices to show that the only $2$-regular $n$-vertex graphs have the form $aC_3 \sqcup bC_4 \sqcup cC_6$ and the only connected $(n-3)$-regular $n$-vertex graphs have the form $(aC_3 \sqcup bC_4 \sqcup cC_6)^{\textup{c}}$.

    The former claim follows from the fact that the only connected $2$-regular integral graphs are $C_3$, $C_4$, and $C_6$ \cite{GM94}. The latter claim follows via complementation: $G$ is integral and $(n-3)$-regular if and only if $G^{\textup{c}}$ is integral and $2$-regular.
\end{proof}

In fact, it is known that, for every integer $r$, there are only finitely many connected $r$-regular integral graphs (on any number of vertices) \cite{Cve75}. It follows that formulas analogous to (but much uglier than) that of Corollary~\ref{cor:prime_reg_count_lap_int_2reg} are possible for $rl(n,r) = crl(n,n-r-1)$ for all fixed $r$. However, even when $r = 3$ (in which case there are $13$ connected $3$-regular integral graphs \cite{BC76}), this formula is extremely unwieldy to write down, and when $r \geq 4$ the characterization of connected $r$-regular integral graphs is not yet complete (see \cite{GKLZ20} and the references therein).

\section{Enumerating regular $\{-1,0,1\}$-diagonalizable graphs}\label{sec:prime_regular_negone}

We now explore what these results look like when further restricted to $\{-1,0,1\}$-diagonalizable regular graphs. To start, we recall that every $\{-1,0,1\}$-diagonalizable graph is Laplacian integral, from which it immediately follows that the only connected $2$-regular $\{-1,0,1\}$-diagonalizable graphs are $C_3$, $C_4$, and $C_6$. Similarly, the known $13$-graph characterization of $3$-regular integral graphs immediately gives the following exhaustive list of connected $3$-regular $\{-1,0,1\}$-diagonalizable graphs:

\begin{proposition}\label{prop:neg_one_3reg}
    There are exactly nine different connected $3$-regular $\{-1,0,1\}$-diagonalizable graphs:
    \begin{center}
        $K_4$, \ \ $K_{3,3}$, \ \ $C_3 \mathbin{\square} K_2$, \ \ $C_4 \mathbin{\square} K_2$, \ \ $C_6 \mathbin{\square} K_2$,\\
        the Petersen graph, the Desargues graph, the Nauru graph, and the Tutte $8$-cage.
    \end{center}
\end{proposition}

\begin{proof}
    Explicit $\{-1,0,1\}$-diagonalizations of these nine graphs are available at \cite{VaronaCode}. There are only four other connected integral $3$-regular graphs \cite{BC76} (and thus only four other connected $3$-regular graphs that could potentially be $\{-1,0,1\}$-diagonalizable); our code showed that they are not $\{-1,0,1\}$-diagonalizable.
\end{proof}

Our computational techniques (see Appendix~\ref{sec:appendix_compute_lapint}) are able to show that there are very few regular $\{-1,0,1\}$-diagonalizable graphs of small order. In particular, there are only $97$ connected regular $\{-1,0,1\}$-diagonalizable graphs on $15$ or fewer vertices---see Table~\ref{tab:reg_negone}.

\begin{table}[!htb]
    \centering
    \begin{tabular}{clc}\toprule
        \# of vertices ($n$) & connected regular $\{-1,0,1\}$-diagonalizable graphs & \# of graphs\\\toprule
        $1$, $2$, $3$, $5$, $7$, $11$, $13$ & $K_n$ only & $1$ \\\midrule
        $4$ & $K_4$, \ $C_4$ & $2$ \\\midrule
        $6$ & $K_6$, \ $K_{3,3}$, \ $K_{2,2,2}$, \ $C_6$, \ $K_3 \mathbin{\square} K_{2}$ & $5$ \\\midrule
        $8$ & $K_8$, \ $K_{4,4}$, \ $K_{2,2,2,2}$, \ $(C_4 \sqcup C_4)^\textup{c}$, \ $K_4 \mathbin{\square} K_2$, \ $(K_4 \mathbin{\square} K_2)^\textup{c}$ & $6$ \\\midrule
        $9$ & $K_9$, \ $K_{3,3,3}$, \ $K_3 \mathbin{\square} K_3$ & $3$ \\\midrule
        $10$ & $K_{10}$, \ $K_{5,5}$, \ $K_{2,2,2,2,2}$, \ $K_5 \mathbin{\square} K_2$, \ $(K_5 \mathbin{\square} K_2)^{\textup{c}}$, \ Petersen graph $P$, \ $P^{\textup{c}}$ & $7$ \\\midrule
        $12$ & Too many to list; see \cite{VaronaCode}. & $46$ \\\midrule
        $14$ & $K_{14}$, \ $K_{7,7}$, \ $K_{2,2,2,2,2,2,2}$, \ $K_7 \mathbin{\square} K_2$, \ $(K_7 \mathbin{\square} K_2)^{\textup{c}}$ & $5$ \\\midrule
        $15$ & $K_{15}$, \ $K_{5,5,5}$, \ $K_{3,3,3,3,3}$, \ $K_5 \mathbin{\square} K_3$, \ $(K_5 \mathbin{\square} K_3)^{\textup{c}}$, \ $\mathrm{Lin}(P)$, \ $(\mathrm{Lin}(P))^{\textup{c}}$, & $16$ \\
        & $(C_6 \sqcup 3K_3)^{\textup{c}}$, \ $\mathrm{Lin}(K_6)$, \ $(\mathrm{Lin}(K_6))^{\textup{c}}$, and 6 others (unnamed). & \\\bottomrule
    \end{tabular}
    \caption{All connected regular $\{-1,0,1\}$-diagonalizable graphs on $15$ or fewer vertices. Note that $\mathrm{Lin}(G)$ denotes the line graph of the graph $G$.}\label{tab:reg_negone}
\end{table}

Based on these computational results, we might be tempted to conjecture that $K_n$ is the only connected regular $\{-1,0,1\}$-diagonalizable graph of order $n$ when $n$ is prime. However, by making use of Corollary~\ref{cor:prime_reg_count_lap_int} (but restricting to just balanced multisets with parts between $r+1$ and $n-r-1$ inclusive, rather than all such integer partitions), we can see that this is false:

\begin{theorem}\label{thm:prime_regular_negone}
    Let $n$ be prime. There exists a regular $n$-vertex connected $\{-1,0,1\}$-diagonalizable graph other than $K_n$ if and only if $n \geq 31$.
\end{theorem}

\begin{proof}
    First, we show that if $n \geq 31$ is prime, then there exists a connected regular graph other than $K_n$ that is $\{-1,0,1\}$-diagonalizable.
    
    To this end, we notice that if $n \equiv 1 \pmod{3}$ then we can write $n = 3a + 4b$ for $a = (n-16)/3 \geq 3$ and $b = 4$, and if $n \equiv 2 \pmod{3}$ then we can write $n = 3a + 4b$ for $a = (n-20)/3 \geq 3$ and $b = 5$ (and $n \equiv 0 \pmod{3}$ is impossible since $n \geq 31 > 3$ is prime). In particular, this means (thanks to Proposition~\ref{prop:ab_balanced}) that the vector $(4\mathbf{1}_a,3\mathbf{1}_b)$ is balanced. It follows from Theorem~\ref{thm:prime_main} that the graph $G = (aC_4 \sqcup bC_3)^\textup{c}$ is $\{-1,0,1\}$-diagonalizable (since $C_3$ and $C_4$ both are). Since $C_3$ and $C_4$ are both regular, $G$ is $(n-3)$-regular.

    Next, we show that if $n < 31$ is prime and $G$ is a connected regular $\{-1,0,1\}$-diagonalizable graph, then $G = K_n$. All connected regular $\{-1,0,1\}$-diagonalizable graphs on $n$ vertices for $n \leq 15$ are tabulated in Table~\ref{tab:reg_negone} and \cite{VaronaCode}, and there are no counterexamples to the statement of the theorem for those values of $n$, so we assume for the remainder of the proof that $n \in \{17,19,23,29\}$.
    
    Suppose that $G^\textup{c}$ is $r$-regular (i.e., $G$ is $(n-r-1)$-regular). By the handshaking lemma, if $r$ is odd then $n$ is even, which contradicts our assumption that $n \in \{17,19,23,29\}$. If $r = 0$ then $G^\textup{c}$ has no edges, so $G = K_n$. All that remains is to rule out the possibility that $r \geq 2$ is even.\medskip

    \underline{Case 1: $r = 2$.} By Corollary~\ref{cor:prime_reg_count_lap_int_2reg}, $G = (aC_3 \sqcup bC_4 \sqcup cC_6)^{\textup{c}}$ for some $(a,b,c) \in \Znn^3$ such that $3a+4b+6c = n$. Furthermore, at least two of $a,b,c$ must be non-zero, since otherwise $n$ is a multiple of $3$, $4$, or $6$ (which contradicts $n \in \{17,19,23,29\}$). Since $(3\mathbf{1}_{a},4\mathbf{1}_{b},6\mathbf{1}_{c})$ must be balanced (by Theorem~\ref{thm:prime_main}), Proposition~\ref{prop:346_balanced} thus tells us that one of the following must happen:
    \begin{itemize}
        \item[a)] $a = 0$, $b \geq 3$, and $c \geq 2$. This implies $n = 3a+4b+6c = 4b+6c$ is even, which contradicts $n \in \{17,19,23,29\}$.

        \item[b)] $a \geq 2$, $b = 0$, $c \geq 1$. This implies $n = 3a+4b+6c = 3a+6c$ is a multiple of $3$, which contradicts $n \in \{17,19,23,29\}$.

        \item[c)] $a \geq 4$, $b \geq 3$, $c = 0$. This implies $n = 3a+4b+6c = 3a+4b \geq 12+12 = 24$, so $n = 29$. However, there does not exist an integer partition of $29$ into parts of size $3$ and $4$, with at least $4$ parts of size $3$ and at least $3$ parts of size $4$ (to see this, notice that any such partition must contain $\{3,3,3,3,4,4,4\}$ as a subset, which has a sum of $24$, and there is no way to use parts of size $3$ and $4$ to add up to the remaining $29 - 24 = 5$).

        \item[d)] $a \geq 2$, $b \geq 3$, $c \geq 1$. This implies $n = 3a+4b+6c \geq 6+12+6 = 24$, so $n = 29$. The same argument as in part~(c) shows that this is not possible: there does not exist a partition of $29$ into parts of size $3$, $4$, and $6$ that contains $\{3,3,4,4,4,6\}$ as a subset.
    \end{itemize}

    \underline{Case 2: $r \geq 4$.} Since each vertex of $G^\textup{c}$ has at least $4$ neighbors, each connected component of $G^\textup{c}$ must have at least $5$ vertices. Since $n \leq 29$ and $\lfloor 29/5 \rfloor = 5$, this implies that $G^{\textup{c}}$ has at most $5$ connected components. When combined with Theorem~\ref{thm:prime_main}, this tells us that
    \[
        G = (G_1 \sqcup G_2 \sqcup \cdots \sqcup G_d)^\textup{c}
    \]
    for some $d \in \{2,3,4,5\}$, where $\mathbf{v} = (|V(G_1)|, \ldots, |V(G_d)|) \in \Zp^d$ is balanced with $|V(G_1)| + \cdots + |V(G_d)| = n$.

    However, all $124$ balanced vectors with $2$, $3$, $4$, or $5$ positive integer entries and greatest common divisor equal to $1$ were tabulated in \cite[Table~1]{johnston2025laplacian}, and none of them have sum in the set $\{17,19,23,29\}$. Since these target values are all prime, no multiple of these $124$ vectors has sum in the set $\{17,19,23,29\}$ either, so $G$ does not exist when $n \in \{17,19,23,29\}$. This completes the $r \geq 4$ case and the proof.  
\end{proof}

When we combine the ideas of the proof of Theorem~\ref{thm:prime_regular_negone} with computer scripts to compute all balanced vectors satisfying the constraints of Proposition~\ref{prop:346_balanced}, we can compute all connected regular $\{-1,0,1\}$-diagonalizable graphs of prime order up to $53$; see Table~\ref{tab:regular_prime_101}.

\begin{longtable}{clc}\toprule
        \# of vertices ($n$) & connected regular $\{-1,0,1\}$-diagonalizable graphs & \# of graphs\\\toprule
        $2$--$29$ (prime) & $K_n$ only & $1$ \\\midrule
        $31$ & $K_{31}$, \ $(5C_3 \sqcup 4C_4)^\textup{c}$, \ $(3C_3 \sqcup 4C_4 \sqcup C_6)^\textup{c}$ & $3$ \\\midrule
        $37$ & $K_{37}$, \ $(7C_3 \sqcup 4C_4)^\textup{c}$, \ $(3C_3 \sqcup 4C_4 \sqcup 2C_6)^\textup{c}$, \ $(5C_3 \sqcup 4C_4 \sqcup C_6)^\textup{c}$ & $4$ \\\midrule
        $41$ & $K_{41}$, \ $(7C_3 \sqcup 5C_4)^\textup{c}$, \ $(3C_3 \sqcup 5C_4 \sqcup 2C_6)^\textup{c}$, \ $(5C_3 \sqcup 5C_4 \sqcup C_6)^\textup{c}$ & $4$ \\\midrule
        $43$ & $K_{43}$, \ $(5C_3 \sqcup 7C_4)^\textup{c}$, \ $(9C_3 \sqcup 4C_4)^\textup{c}$, \ $(3C_3 \sqcup 4C_4 \sqcup 3C_6)^\textup{c}$, & $7$\\
        & $(3C_3 \sqcup 7C_4 \sqcup C_6)^\textup{c}$, \ $(5C_3 \sqcup 4C_4 \sqcup 2C_6)^\textup{c}$, \ $(7C_3 \sqcup 4C_4 \sqcup C_6)^\textup{c}$ & \\\midrule
        $47$ & $K_{47}$, \ $(5C_3 \sqcup 8C_4)^\textup{c}$, \ $(9C_3 \sqcup 5C_4)^\textup{c}$, \ $(3C_3 \sqcup 5C_4 \sqcup 3C_6)^\textup{c}$, & $10$ \\
        & $(3C_3 \sqcup 8C_4 \sqcup C_6)^\textup{c}$, \ $(5C_3 \sqcup 5C_4 \sqcup 2C_6)^\textup{c}$, \ $(7C_3 \sqcup 5C_4 \sqcup C_6)^\textup{c}$, & \\
        & $(2K_5 \sqcup 2K_{2,2,2} \sqcup (K_3 \mathbin{\square} K_3) \sqcup 2K_{4,4})^\textup{c}$, & \\
        & $(2K_5 \sqcup 2K_{2,2,2} \sqcup (K_3 \mathbin{\square} K_3) \sqcup 2(K_4 \mathbin{\square} K_2))^\textup{c}$, & \\
        & $(2K_5 \sqcup 2K_{2,2,2} \sqcup (K_3 \mathbin{\square} K_3) \sqcup K_{4,4} \sqcup (K_4 \mathbin{\square} K_2))^\textup{c}$ & \\\midrule
        $53$ & $K_{53}$, $(7C_3 \sqcup 8C_4)^\textup{c}$, $(11C_3 \sqcup 5C_4)^\textup{c}$, $(3C_3 \sqcup 5C_4 \sqcup 4C_6)^\textup{c}$, & $62$ \\
        & $(3C_3 \sqcup 8C_4 \sqcup 2C_6)^\textup{c}$, $(5C_3 \sqcup 5C_4 \sqcup 3C_6)^\textup{c}$, $(5C_3 \sqcup 8C_4 \sqcup C_6)^\textup{c}$, & \\
        & $(7C_3 \sqcup 5C_4 \sqcup 2C_6)^\textup{c}$, $(9C_3 \sqcup 5C_4 \sqcup C_6)^\textup{c}$, & \\
        & $(K_5 \sqcup 2K_{2,2,2} \sqcup K_{4,4} \sqcup 2(K_3 \mathbin{\square} K_3) \sqcup (K_5 \mathbin{\square} K_2)^\textup{c})^\textup{c}$, & \\
        & $(K_5 \sqcup 2K_{2,2,2} \sqcup (K_4 \mathbin{\square} K_2) \sqcup 2(K_3 \mathbin{\square} K_3) \sqcup (K_5 \mathbin{\square} K_2)^\textup{c})^\textup{c}$, & \\
        & $(2K_5 \sqcup K_{2,2,2} \sqcup 2K_{4,4} \sqcup (K_3 \mathbin{\square} K_3) \sqcup G_{12,4})^\textup{c}$, & \\
        & $(2K_5 \sqcup K_{2,2,2} \sqcup 2(K_4 \mathbin{\square} K_2) \sqcup (K_3 \mathbin{\square} K_3) \sqcup G_{12,4})^\textup{c}$, & \\
        & $(2K_5 \sqcup K_{2,2,2} \sqcup K_{4,4} \sqcup (K_4 \mathbin{\square} K_2) \sqcup (K_3 \mathbin{\square} K_3) \sqcup G_{12,4})^\textup{c}$, & \\
        & $(2K_5 \sqcup 2K_{2,2,2} \sqcup 2K_{4,4} \sqcup \mathrm{Lin}(P))^\textup{c}$, & \\
        & $(2K_5 \sqcup 2K_{2,2,2} \sqcup 2(K_4 \mathbin{\square} K_2) \sqcup \mathrm{Lin}(P))^\textup{c}$, & \\
        & $(2K_5 \sqcup 2K_{2,2,2} \sqcup K_{4,4} \sqcup (K_4 \mathbin{\square} K_2) \sqcup \mathrm{Lin}(P))^\textup{c}$, & \\
        & $(2K_5 \sqcup 2K_{2,2,2} \sqcup (K_3 \mathbin{\square} K_3) \sqcup (K_5 \mathbin{\square} K_2)^\textup{c} \sqcup G_{12,4})^\textup{c}$, & \\
        & $(2K_5 \sqcup 3K_{2,2,2} \sqcup 2K_{4,4} \sqcup (K_3 \mathbin{\square} K_3))^\textup{c}$, & \\
        & $(2K_5 \sqcup 3K_{2,2,2} \sqcup 2(K_4 \mathbin{\square} K_2) \sqcup (K_3 \mathbin{\square} K_3))^\textup{c}$, & \\
        & $(2K_5 \sqcup 3K_{2,2,2} \sqcup K_{4,4} \sqcup (K_4 \mathbin{\square} K_2) \sqcup (K_3 \mathbin{\square} K_3))^\textup{c}$, & \\
        & $(2K_5 \sqcup 4K_{2,2,2} \sqcup (K_3 \mathbin{\square} K_3) \sqcup (K_5 \mathbin{\square} K_2)^\textup{c})^\textup{c}$, & \\
        & $(3K_5 \sqcup 2K_{2,2,2} \sqcup K_{4,4} \sqcup 2(K_3 \mathbin{\square} K_3))^\textup{c}$, & \\
        & $(3K_5 \sqcup 2K_{2,2,2} \sqcup (K_4 \mathbin{\square} K_2) \sqcup 2(K_3 \mathbin{\square} K_3))^\textup{c}$, & \\
        & $(3K_5 \sqcup 2K_{2,2,2} \sqcup 2K_{4,4} \sqcup (K_5 \mathbin{\square} K_2)^\textup{c})^\textup{c}$, & \\
        & $(3K_5 \sqcup 2K_{2,2,2} \sqcup 2(K_4 \mathbin{\square} K_2) \sqcup (K_5 \mathbin{\square} K_2)^\textup{c})^\textup{c}$, & \\
        & $(3K_5 \sqcup 2K_{2,2,2} \sqcup K_{4,4} \sqcup (K_4 \mathbin{\square} K_2) \sqcup (K_5 \mathbin{\square} K_2)^\textup{c})^\textup{c}$, & \\
        & $(4K_5 \sqcup 2K_{2,2,2} \sqcup (K_3 \mathbin{\square} K_3) \sqcup G_{12,4})^\textup{c}$, & \\
        & $(4K_5 \sqcup 4K_{2,2,2} \sqcup (K_3 \mathbin{\square} K_3))^\textup{c}$, & \\
        & $(5K_5 \sqcup 2K_{2,2,2} \sqcup 2K_{4,4})^\textup{c}$, & \\
        & $(5K_5 \sqcup 2K_{2,2,2} \sqcup 2(K_4 \mathbin{\square} K_2))^\textup{c}$, & \\
        & $(5K_5 \sqcup 2K_{2,2,2} \sqcup K_{4,4} \sqcup (K_4 \mathbin{\square} K_2))^\textup{c}$ & \\\bottomrule
    \caption{A list of all connected regular $\{-1,0,1\}$-diagonalizable graphs of small prime orders. Each $53$-vertex graph that contains $G_{12,4}$ in its decomposition is actually $7$ different graphs; $G_{12,4}$ can be any of the $7$ different connected $12$-vertex $4$-regular $\{-1,0,1\}$-diagonalizable graphs (see \cite{VaronaCode}). The line graph of the Petersen graph is denoted by $\mathrm{Lin}(P)$.}\label{tab:regular_prime_101}
\end{longtable}

It is perhaps worth noting that all prime-order connected regular $\{-1,0,1\}$-diagonalizable graphs on $29$ or fewer vertices are $K_n$, which can be seen as a degree condition: all such graphs on $29$ or fewer vertices have degree $n-1$. Similarly, all such prime-order graphs on $43$ or fewer vertices have degree at least $n-3$, and all such prime-order graphs on $59$ or fewer vertices have degree at least $n-5$. In general, the restriction that the sizes of the connected components form a balanced vector makes it so that the degree of a connected regular $\{-1,0,1\}$-diagonalizable graph cannot be much smaller than the order of the graph.

\section{$\{-1,0,1\}$-diagonalizable graphs of large order}\label{sec:asymp}

We now consider the problem of bounding the number of connected $\{-1,0,1\}$-diagonalizable graphs when $n$ (the number of vertices) is large. Following the notation of Corollary~\ref{cor:prime_count_nozo_int_b}, we let $cs(n)$ denote the number of connected $\{-1,0,1\}$-diagonalizable graphs of order $n$. The subexponential lower bound $cs(n) \geq 13^{\sqrt{n}}$ when $n$ is sufficiently large was proved in \cite[Corollary~4]{johnston2025laplacian}. We now improve this result by proving an exponential lower bound:

\begin{theorem}\label{thm:nozo_lower_bound}
    For all sufficiently large $n$, we have $cs(n) \geq (7/4)^n$.
\end{theorem}

\begin{proof}
    Let $f : \Zp \rightarrow \Zp$ be the function satisfying the recurrence~\eqref{eq:cu_recurrence} for all (not just prime) $n \geq 14$, and which has $f(n) = cs(n)$ for $n \in \{1,2,\ldots,13\}$. Then $cs(n) \geq f(n)$ for all $n \in \Zp$ since $f$ counts the number of connected $\{-1,0,1\}$-diagonalizable graphs of a certain type (i.e., the ones that can be written in the form $(G_1 \sqcup G_2 \sqcup \cdots \sqcup G_d)^\textup{c}$ for some $d \geq 2$ and connected $\{-1,0,1\}$-diagonalizable graphs $G_1, G_2, \ldots, G_d$ either also being of this type or being on at most $13$ vertices). We will show that $f(n) \geq (7/4)^n$ for all sufficiently large $n$, from which the theorem follows.

    To start, fix $\ell \in \Zp$. Then the recurrence~\eqref{eq:cu_recurrence} tells us that
    \[
        f(2\ell) = \sum_{p \in B_{2\ell} \setminus \{2\ell\}} \prod_{k=1}^{\ell} \binom{f(k)+m_p(k)-1}{m_p(k)}.
    \]
    Since the multiset $p = \{\ell,\ell\} \in B_{2\ell} \setminus \{2\ell\}$ is balanced (via the matrix $A = [1 \ \ -1]$) and has
    \[
        \prod_{k=1}^{\ell} \binom{f(k)+m_p(k)-1}{m_p(k)} = \binom{f(\ell)+2-1}{2} = \binom{f(\ell)+1}{2} \geq \frac{(f(\ell))^2}{2},
    \]
    it follows that $f(2\ell) \geq (f(\ell))^2 / 2$. Iterating this recurrence shows that
    \begin{align}\label{ineq:nonz_2sl_bound}
        f(2^s\ell) \geq 2\left(\frac{f(\ell)}{2}\right)^{2^s}
    \end{align}
    for all $s \in \Zp$. This gets us an exponential lower bound on $f(n)$ when $n = 2^s\ell$ for some $s \in \Zp$; our next goal is to generalize this exponential lower bound to \emph{all} sufficiently large $n$.

    To this end, suppose $n \geq 4\ell^2$, let $s \in \Zp$ be the largest integer for which $n \geq 4^s\ell^2$, and define $t = 2^s\ell$. Write $n = qt + r$, where $q,r \in \Zp$ are such that $t \leq r < 2t$ (i.e., $q = \lfloor n/t \rfloor - 1$ and $r = n - qt$). Then recurrence~\eqref{eq:cu_recurrence} tells us that
    \begin{align}\label{eq:nozo_rec_again}
        f(n) = \sum_{p \in B_{n} \setminus \{n\}} \prod_{k=1}^{\lfloor n/2 \rfloor} \binom{f(k)+m_p(k)-1}{m_p(k)}.
    \end{align}
    Proposition~\ref{prop:ab_balanced} tells us that $p = (t\mathbf{1}_q,\mathbf{1}_r)$ (written in vector form instead of multiset form) is balanced. The sum of its entries is $qt + r = n$, and the associated term in the recurrence~\eqref{eq:nozo_rec_again} is
    \begin{align}\label{ineq:nonz_binom_bound}
        \prod_{k=1}^{\lfloor n/2 \rfloor} \binom{f(k)+m_p(k)-1}{m_p(k)} = \binom{f(t)+q-1}{q}\binom{f(1)+r-1}{r} = \binom{f(t)+q-1}{q} \geq \frac{(f(t))^q}{q!},
    \end{align}
    where the final inequality follows from the fact that $(f(t) + q - 1)! / (f(t) - 1)! \geq (f(t))^q$. Combining Inequalities~\eqref{ineq:nonz_2sl_bound} and~\eqref{ineq:nonz_binom_bound} with the recurrence~\eqref{eq:nozo_rec_again} shows that
    \begin{align}\label{ineq:nonz_combined}
        f(n) \geq \frac{(f(t))^q}{q!} \geq \frac{\big(2(f(\ell) / 2)^{2^s}\big)^q}{q!} = \frac{2^q(f(\ell)/2)^{2^s q}}{q!}.
    \end{align}

    Since $s$ is the largest integer for which $n \geq 4^s\ell^2 = t^2$, we have $\sqrt{n}/2 < t \leq \sqrt{n}$. Since $q = \lfloor n/t \rfloor - 1$, this implies $\sqrt{n} - 2 \leq q \leq \lfloor 2\sqrt{n} \rfloor$. Similarly, $q \geq n/t - 2$, so multiplying through by $2^s = t/\ell$ gives $2^s q \geq (n - 2t)/\ell \geq (n - 2\sqrt{n})/\ell$. Plugging these bounds on $q$ and $2^s q$ into Inequality~\eqref{ineq:nonz_combined} gives
    \begin{align}\label{eq:nonz_lb_gell_cell}
        f(n) \geq \frac{2^{\sqrt{n} - 2}}{(\lfloor 2\sqrt{n}\rfloor)!} \left(\sqrt[\ell]{f(\ell)/2}\right)^{n - 2\sqrt{n}} = g_{\ell}(n)c_\ell^n,
    \end{align}
    where
    \[
        g_{\ell}(n) = \frac{2^{\sqrt{n} - 2}}{c_{\ell}^{2\sqrt{n}}(\lfloor 2\sqrt{n}\rfloor)!} \quad \text{and} \quad c_\ell = \sqrt[\ell]{f(\ell)/2}.
    \]
    By using the approximation $(\lfloor 2\sqrt{n}\rfloor)! \leq (2\sqrt{n})^{2\sqrt{n}}$ and taking logarithms, we see that
    \begin{align}\label{ineq:log_gell}
        \ln\big(g_{\ell}(n)\big) \geq (\sqrt{n} - 2)\ln(2) - 2\sqrt{n}\ln(c_{\ell}) - 2\sqrt{n}\ln(2\sqrt{n}).
    \end{align}
    In particular, since $\sqrt{n}\ln(n)$ grows sub-linearly, Inequality~\eqref{ineq:log_gell} implies that, for all $\delta > 0$, we have $\ln(g_{\ell}(n)) \geq -\delta n$ for sufficiently large $n$. Taking exponentials of both sides then shows that
    \[
        g_{\ell}(n) \geq \big(e^{-\delta}\big)^n
    \]
    for all sufficiently large $n$. Since $e^{-\delta} < 1$ but can be made arbitrarily close to $1$ (by choosing $\delta$ sufficiently close to $0$), plugging back into Inequality~\eqref{eq:nonz_lb_gell_cell} then shows that, for all $0 < b < \sqrt[\ell]{f(\ell)/2}$,\footnote{In particular, the conversion between $\delta$ and $b$ is $b = e^{-\delta}\sqrt[\ell]{f(\ell)/2}$.} we have
    \begin{align}\label{ineq:nonz_general_exp_lb}
        f(n) \geq b^n
    \end{align}
    for all sufficiently large $n$. This demonstrates exponential growth of $f(n)$ (and thus $cs(n)$) as long as $\ell$ is chosen so that $\sqrt[\ell]{f(\ell)/2} > 1$ (i.e., $f(\ell) > 2$).

    It now suffices to compute $f(\ell)$ for some choice of $\ell$. When $\ell = 100$, we computed all members of $B_{100}$ and then used the recursion~\eqref{eq:nozo_rec_again} to compute
    \[
        f(100) = 9,092,715,187,206,430,468,500,965.
    \]
    It follows that $f(n) \geq b^n$ for every $b$ satisfying $0 < b < \sqrt[100]{f(100)/2} \approx 1.76$. In particular, we can choose $b = 7/4 = 1.75$ to complete the proof.
\end{proof}

We note that the base $7/4$ in Theorem~\ref{thm:nozo_lower_bound} is not optimal, and a more careful (and computationally expensive) argument could likely increase it a fair bit. In particular, we used $\ell = 100$ at the end of the proof because large values of $\ell$ seem to result in larger values of the base $\sqrt[\ell]{f(\ell)/2}$. If we just used $\ell = 12$, for example, then we would have gotten a weaker theorem: $f(12) = 263$ (see Table~\ref{tab:graph_counts_thirteen}), so $f(n) \geq b^n$ (when $n$ is sufficiently large) for every $b$ satisfying $0 < b < \sqrt[12]{f(12)/2} \approx 1.501679$.

Since $\{-1,0,1\}$-diagonalizable graphs are Laplacian integral, Theorem~\ref{thm:nozo_lower_bound} immediately implies that $cl(n) \geq (7/4)^n$ when $n$ is large as well. However, a much better bound on the number of connected Laplacian integral graphs presents itself when we consider that every cograph is Laplacian integral, and that the number of cographs (and thus the number of connected Laplacian integral graphs) of order $n$ is bounded below by $3.5608^n$ when $n$ is large \cite{ravelomanana2001asymptotic}.

\section{Bipartite $\{-1,0,1\}$-diagonalizable graphs}\label{sec:bipartite}

Our final main result shows that connected $\{-1,0,1\}$-diagonalizable graphs that are bipartite must be regular. This drastically reduces the search space for connected bipartite $\{-1,0,1\}$-diagonalizable graphs.

\begin{theorem}\label{thm:bipartite_imp_regular}
	Let $G$ be a connected bipartite graph. If $G$ is $\{-1,0,1\}$-diagonalizable, then $G$ is regular.
\end{theorem}

\begin{proof}
	Let $V(G) = A \sqcup B$ be a bipartition of $G$, where $|A| = m_1$ and $|B| = m_2$ (so that $|V(G)| = m_1 + m_2$). Let $\mathbf{z} = (z_1, z_2, \ldots, z_{m_1+m_2})$, where
	\[
		z_i = \begin{cases}
			1, & i \in A,\\
			-1, & i \in B.
		\end{cases}
	\]
	Our first goal is to show that $\mathbf{z}$ is an eigenvector of the Laplacian matrix $L(G) = D(G) - A(G)$.
	
	To this end, let $Z = \operatorname{diag}(\mathbf{z})$. Since $G$ is bipartite, every edge has one endpoint in $A$ and the other in $B$. It follows that $ZD(G)Z = D(G)$ and $ZA(G)Z = -A(G)$, so $Q := ZL(G)Z = Z(D(G) - A(G))Z = D(G) + A(G)$ is the signless Laplacian of $G$. Therefore, $Q$ is similar to $L(G)$, so $Q$ and $L(G)$ have the same eigenvalues with the same algebraic multiplicities.

	The matrix $Q$ is entrywise nonnegative. Since $G$ is connected, $Q$ is irreducible, so the Perron--Frobenius theorem tells us that its spectral radius $\lambda:=\rho(Q)$ is a simple eigenvalue of $Q$, and it has a corresponding eigenvector $\mathbf{w}$ whose entries are all strictly positive. Thus
	\[
		Q\mathbf{w}=\lambda\mathbf{w}.
	\]
	Using $Q=ZL(G)Z$, we get
	\[
		ZL(G)Z\mathbf{w}=\lambda\mathbf{w}.
	\]
	Multiplying by $Z$ and using $Z^2 = I$ gives
	\[
		L(G)(Z\mathbf{w})=\lambda(Z\mathbf{w}).
	\]
	Hence
	\[
		\mathbf{v} := Z\mathbf{w}
	\]
	is an eigenvector of $L$ with eigenvalue $\lambda$. Moreover, $v_i > 0$ when $i \in A$ and $v_i < 0$ when $i \in B$.

	Since $G$ is $\{-1,0,1\}$-diagonalizable, there exists an invertible matrix $P$ with entries in $\{-1,0,1\}$ and a diagonal matrix $D$ such that
	\[
		L(G) = PDP^{-1}.
	\]
	Since $\lambda$ is a simple eigenvalue of $Q$, and hence also a simple eigenvalue of $L(G)$, the $\lambda$-eigenspace of $L(G)$ is one-dimensional. Therefore, the column of $P$ corresponding to $\lambda$ must be a scalar multiple of $\mathbf{v}$. Let this column be $\mathbf{p}$. Since $\mathbf{p}$ is a scalar multiple of $\mathbf{v}$, which has no entries equal to $0$, we see that $\mathbf{p}$ also has no entries equal to $0$. Since every entry of $\mathbf{p}$ belongs to $\{-1,0,1\}$, but no entry equals $0$, we conclude that in fact every entry of $\mathbf{p}$ belongs to $\{-1,1\}$. Because $\mathbf{v}$ is positive on $A$ and negative on $B$, it follows that either
	\[
		\mathbf{p} = \mathbf{z} \quad \text{or} \quad \mathbf{p} = -\mathbf{z}.
	\]
	It follows that $\mathbf{z}$ is an eigenvector of $L(G)$ with corresponding eigenvalue $\lambda$, and we have proven our first claim.

	Now notice that if $i \in A$, then
	\[
		[L(G)\mathbf{z}]_i = \deg(i)z_i-\sum_{j\sim i} z_j = \deg(i) + \deg(i) = 2\deg(i).
	\]
	Since $L(G)\mathbf{z} = \lambda\mathbf{z}$ and $z_i = 1$, this implies $\lambda = 2\deg(i)$ for all $i \in A$.
	
	Similarly, if $i \in B$, then
	\[
		[L(G)\mathbf{z}]_i = -\deg(i)-\deg(i) = -2\deg(i).
	\]
	Since $L(G)\mathbf{z} = \lambda\mathbf{z}$ and $z_i = -1$, this implies $\lambda = 2\deg(i)$ for all $i \in B$. It follows that every vertex of $G$ has degree $\lambda/2$, so $G$ is regular.
\end{proof}

\begin{corollary}\label{cor:odd_bip}
    There is a connected bipartite $\{-1,0,1\}$-diagonalizable graph on $n$ vertices if and only if $n$ is even.
\end{corollary}

\begin{proof}
    For the ``if'' direction, simply notice that if $n$ is even then $K_{n/2,n/2}$ is $\{-1,0,1\}$-diagonalizable.
    
    For the ``only if'' direction, suppose $n \geq 3$ and $G$ is a connected bipartite $\{-1,0,1\}$-diagonalizable graph. By Theorem~\ref{thm:bipartite_imp_regular}, since $G$ is bipartite it must be regular. Since $G$ is both bipartite and regular, both parts of $G$ must have the same number of vertices, so $n$ must be even, which completes the proof.
\end{proof}

We note that neither of these results about bipartite graphs holds if the hypothesis that the graph is $\{-1,0,1\}$-diagonalizable is weakened to just Laplacian integral. For example, the path graph $P_3$ is connected, bipartite, and Laplacian integral (the eigenvalues of its Laplacian are $0$, $1$, and $3$), but it is not regular.

When we combine Theorem~\ref{thm:bipartite_imp_regular}, Corollary~\ref{cor:odd_bip}, and our computational methods of Appendix~\ref{sec:appendix_compute_lapint}, we are able to compute all connected bipartite $\{-1,0,1\}$-diagonalizable graphs of order up to $16$ (see Table~\ref{tab:bipartite_negone}).

\begin{table}[!htb]
    \centering
    \begin{tabular}{clc}\toprule
        \# of vertices ($n$) & connected bipartite $\{-1,0,1\}$-diagonalizable graphs & \# of graphs\\\toprule
        $2$ & $K_2$ & $1$ \\\midrule
        $4$ & $K_{2,2}$ & $1$ \\\midrule
        $6$ & $K_{3,3}$, \ $(K_3 \mathbin{\square} K_2)^\textup{c}$ & $2$ \\\midrule
        $8$ & $K_{4,4}$, \ $(K_4 \mathbin{\square} K_2)^\textup{c}$ & $2$ \\\midrule
        $10$ & $K_{5,5}$, \ $(K_5 \mathbin{\square} K_2)^\textup{c}$ & $2$ \\\midrule
        $12$ & $K_{6,6}$, \ $(K_6 \mathbin{\square} K_2)^\textup{c}$, \ $K_{3,3} \mathbin{\square} K_2$, \ $C_6 \mathbin{\square} K_2$, \ $C_{12}(1,5)$ & $5$ \\\midrule
        $14$ & $K_{7,7}$, \ $(K_7 \mathbin{\square} K_2)^\textup{c}$ & $2$ \\\midrule
        $16$ & $K_{8,8}$, \ $(K_8 \mathbin{\square} K_2)^\textup{c}$, \ $K_{4,4} \mathbin{\square} K_2$, \ $K_{2,2} \mathbin{\square} K_4$, \ $Q_4$, \ the Hoffman graph & $6$ \\\bottomrule
    \end{tabular}
    \caption{All connected bipartite $\{-1,0,1\}$-diagonalizable graphs on $16$ or fewer vertices. We note that $Q_4 = K_2 \mathbin{\square} K_2 \mathbin{\square} K_2 \mathbin{\square} K_2$ is the tesseract graph and $C_{12}(1,5)$ is the circulant $(1,5)$-graph (see \cite{HoGC1215}).}\label{tab:bipartite_negone}
\end{table}

\section{Conclusions and open questions}\label{sec:conclude}

We have provided computational results that significantly expand what is known about Laplacian integral and $\{-1,0,1\}$-diagonalizable graphs, most notably providing a decomposition theorem (Theorem~\ref{thm:prime_main}) for Laplacian integral graphs of prime order, and verifying that the $S_{n,n}$ conjecture holds for $n = 12$. Some open questions that remain after (and arise from) our work include:

\begin{itemize}
    \item The $n \in \{13,14,15\}$ cases of the $S_{n,n}$ conjecture are already known to hold \cite{fallat2005graphs}. This leaves the $n = 16$ case as the smallest one that is unknown; does the conjecture hold for $n = 16$?

    \item Since an order-$n$ connected graph $G$ being Laplacian integral implies that $(G^\textup{c} \sqcup K_1)^{\textup{c}}$ is an order-$(n+1)$ connected Laplacian integral graph, the function $cl(n)$ is monotonically non-decreasing. Does there exist a positive integer $N$ for which $cs(n)$ is monotonically non-decreasing when $n \geq N$? The data provided in Table~\ref{tab:graph_counts_thirteen} shows that if the answer is ``yes'' then $N \geq 13$.

    \item Our computational results show that \cite[Conjecture~5.4]{HT23} is false, since there are two counterexamples when $n = 9$. However, there are no counterexamples when $n \in \{10,11,12,13\}$, so it seems natural to adjust the conjecture to include the requirement that $n \geq 10$; is it true if this change is made?

    \item Is it NP-hard to determine whether a given vector is balanced? If so, determining $\{-1,0,1\}$-diagonalizability of a graph is also NP-hard via \cite[Corollary~3]{johnston2025laplacian}.
\end{itemize}

\section*{Acknowledgements}

The authors thank Matthew Betti and the Digital Research Alliance of Canada (\href{https://alliancecan.ca}{alliancecan.ca}) for providing computational resources that were necessary to carry out this work. L.M.B.V.\ thanks Benjamin Talbot and Liam Keliher for their help in devising the computational methods described in Appendix~\ref{sec:appendix_compute_lapint}, as well as Craig Brett for further valuable discussions regarding spectral graph theory. N.J.\ was supported by NSERC Discovery Grant number RGPIN-2022-04098. S.P.\ was supported by NSERC Discovery Grant number 1174582, the Canada Foundation for Innovation (CFI) grant number 35711, and the Canada Research Chairs (CRC) Program grant number 231250.

\bibliographystyle{alpha}
\bibliography{ref}

\newcommand{\etalchar}[1]{$^{#1}$}
\begin{thebibliography}{dLDVMT26}

\bibitem[AAF{\etalchar{+}}21]{adm2021weakly}
M.~Adm, K.~Almuhtaseb, S.~Fallat, K.~Meagher, S.~Nasserasr, M.~N. Shirazi, and A.~S. Razafimahatratra.
\newblock Weakly {H}adamard diagonalizable graphs.
\newblock {\em Linear Algebra and its Applications}, 610:86--119, 2021.

\bibitem[BBF{\etalchar{+}}22]{Breen22}
J.~Breen, S.~Butler, M.~Fuentes, B.~Lidick{\`y}, M.~Phillips, A.~W.~N. Riasanovksy, S.-Y. Song, R.~R. Villagr{\'a}n, C.~Wiseman, and X.~Zhang.
\newblock Hadamard diagonalizable graphs of order at most 36.
\newblock {\em The Electronic Journal of Combinatorics}, 29:P2.16, 2022.

\bibitem[BC76]{BC76}
B.~C. Bussemaker and D.~M. Cvetkovi{\'c}.
\newblock There are exactly 13 connected, cubic, integral graphs.
\newblock {\em Univ. Beograd. Publ. Elektrotehn. Fak. Ser. Mat. Fiz.}, (576):43--48, 1976.

\bibitem[BCR{\etalchar{+}}02]{balinska2002survey}
K.~T. Bali{\'n}ska, D.~Cvetkovi{\'c}, Z.~Radosavljevi{\'c}, S.~K. Simi{\'c}, and D.~Stevanovi{\'c}.
\newblock A survey on integral graphs.
\newblock {\em Publikacije Elektrotehni\v{c}kog fakulteta. Serija Matematika}, pages 42--65, 2002.

\bibitem[BFK11]{BFK11}
S.~Barik, S.~Fallat, and S.~Kirkland.
\newblock On {H}adamard diagonalizable graphs.
\newblock {\em Linear Algebra and its Applications}, 435:1885--1902, 2011.

\bibitem[BKSZ01]{BKSZ01}
K.~T. Bali{\'n}ska, M.~Kupczyk, S.~K. Simi{\'c}, and K.~T. Zwierzy{\'n}ski.
\newblock On generating all integral graphs on 12 vertices.
\newblock {\em Technical University of Pozna{\'n}, Computer Science Center Report}, 482, 2001.

\bibitem[BSZ09]{BSZ09}
K.~T. Bali{\'n}ska, S.~K. Simi{\'c}, and K.~T. Zwierzy{\'n}ski.
\newblock Some properties of integral graphs on 13 vertices.
\newblock {\em Technical University of Pozna{\'n}, Computing Science Centre Report}, 578, 2009.

\bibitem[CDG23]{HoGC1215}
K.~Coolsaet, S.~D'hondt, and J.~Goedgebeur.
\newblock Circulant {C}\_12(1,5) at the \emph{House of Graphs}.
\newblock \url{https://houseofgraphs.org/graphs/32806}, 2023.
\newblock [Online; accessed June 25, 2026].

\bibitem[CKK19]{CKK19}
J.-G. Caputo, I.~Khames, and A.~Knippel.
\newblock On graph {L}aplacian eigenvectors with components in $\{-1,0,1\}$.
\newblock {\em Discrete Applied Mathematics}, 269(C):120--129, 2019.

\bibitem[CSG05]{CSG05}
A.~Caprara and J.-J. Salazar-Gonz{\'a}lez.
\newblock Laying out sparse graphs with provably minimum bandwidth.
\newblock {\em INFORMS Journal on Computing}, 17(3):356--373, 2005.

\bibitem[Cve75]{Cve75}
D.~M. Cvetkovi{\'c}.
\newblock Cubic integral graphs.
\newblock {\em Univ. Beograd. Publ. Elektrotehn. Fak. Ser. Mat. Fiz.}, (498--541):107--113, 1975.

\bibitem[dLDVMT26]{Mon26}
L.~de~Lima, R.~Del-Vecchio, H.~Monterde, and H.~Teixeira.
\newblock Structured eigenbases and pair state transfer on threshold graphs.
\newblock {\em arXiv preprint arXiv:2601.13318 [math.CO]}, 2026.

\bibitem[Fie73]{fiedler1973algebraic}
M.~Fiedler.
\newblock Algebraic connectivity of graphs.
\newblock {\em Czechoslovak Mathematical Journal}, 23(2):298--305, 1973.

\bibitem[FKMN05]{fallat2005graphs}
S.~M. Fallat, S.~J. Kirkland, J.~J. Molitierno, and M.~Neumann.
\newblock On graphs whose {L}aplacian matrices have distinct integer eigenvalues.
\newblock {\em Journal of Graph Theory}, 50(2):162--174, 2005.

\bibitem[GKLZ20]{GKLZ20}
S.~Goryainov, E.~V. Konstantinova, H.~Li, and D.~Zhao.
\newblock Integral graphs obtained by dual seidel switching.
\newblock {\em Linear Algebra and its Applications}, 604:476--489, 2020.

\bibitem[GM94]{GM94}
R.~Grone and R.~Merris.
\newblock The {L}aplacian spectrum of a graph {II}.
\newblock {\em SIAM Journal on Discrete Mathematics}, 7(2):221--229, 1994.

\bibitem[GMS90]{grone1990laplacian}
R.~Grone, R.~Merris, and V.~S. Sunder.
\newblock The {L}aplacian spectrum of a graph.
\newblock {\em SIAM Journal on Matrix Analysis and Applications}, 11(2):218--238, 1990.

\bibitem[God11]{God11}
C.~Godsil.
\newblock Periodic graphs.
\newblock {\em The Electronic Journal of Combinatorics}, 18(1), 2011.

\bibitem[HKT22]{HKT22}
A.~Hameed, Z.~U. Khan, and M.~Tyaglov.
\newblock Laplacian energy and first {Z}agreb index of {L}aplacian integral graphs.
\newblock {\em Analele \c{s}tiin\c{t}ifice ale Universit\u{a}\c{t}ii "Ovidius" Constan\c{t}a. Seria Matematic\u{a}}, 30(2):133--160, 2022.

\bibitem[HP73]{HP73}
F.~Harary and E.~M. Palmer.
\newblock {\em Graphical Enumeration}.
\newblock Academic Press, 1973.

\bibitem[HS06]{harary2006graphs}
F.~Harary and A.~J. Schwenk.
\newblock Which graphs have integral spectra?
\newblock In {\em Graphs and Combinatorics: Proceedings of the Capital Conference on Graph Theory and Combinatorics at the George Washington University June 18--22, 1973}, pages 45--51. Springer, 2006.

\bibitem[HT23]{HT23}
A.~Hameed and M.~Tyaglov.
\newblock Integral {L}aplacian graphs with a unique double {L}aplacian eigenvalue, {I}.
\newblock {\em Special Matrices}, 1(1):20230111, 2023.

\bibitem[HT25]{HT25}
A.~Hameed and M.~Tyaglov.
\newblock Integral {L}aplacian graphs with a unique double {L}aplacian eigenvalue, {II}.
\newblock {\em Bulletin of the Korean Mathematical Society}, 62(2):299--317, 2025.

\bibitem[JKP{\etalchar{+}}17]{johnston2017}
N.~Johnston, S.~Kirkland, S.~Plosker, R.~Storey, and X.~Zhang.
\newblock Perfect quantum state transfer using {H}adamard diagonalizable graphs.
\newblock {\em Linear Algebra and its Applications}, 531:375--398, 2017.

\bibitem[Joh23a]{oeisA363064}
N.~Johnston.
\newblock Sequence {A363064} in \emph{{T}he {O}n-{L}ine {E}ncyclopedia of {I}nteger {S}equences}.
\newblock \url{https://oeis.org/A363064}, 2023.
\newblock Number of connected Laplacian integral graphs on n vertices. [Online; accessed June 18, 2026].

\bibitem[Joh23b]{oeisA363065}
N.~Johnston.
\newblock Sequence {A363065} in \emph{{T}he {O}n-{L}ine {E}ncyclopedia of {I}nteger {S}equences}.
\newblock \url{https://oeis.org/A363065}, 2023.
\newblock Number of Laplacian integral graphs on n vertices. [Online; accessed June 30, 2026].

\bibitem[JP25]{johnston2025laplacian}
N.~Johnston and S.~Plosker.
\newblock Laplacian $\{-1, 0, 1\}$-and $\{-1, 1\}$-diagonalizable graphs.
\newblock {\em Linear Algebra and its Applications}, 704:309--339, 2025.

\bibitem[Mer94a]{merris1994degree}
R.~Merris.
\newblock Degree maximal graphs are {L}aplacian integral.
\newblock {\em Linear Algebra and its Applications}, 199:381--389, 1994.

\bibitem[Mer94b]{merris1994laplacian}
R.~Merris.
\newblock Laplacian matrices of graphs: a survey.
\newblock {\em Linear Algebra and its Applications}, 197:143--176, 1994.

\bibitem[Mer98]{Mer98}
R.~Merris.
\newblock Laplacian graph eigenvectors.
\newblock {\em Linear Algebra and its Applications}, 278:221--236, 1998.

\bibitem[MMP25]{MMP25}
D.~McLaren, H.~Monterde, and S.~Plosker.
\newblock Weakly {H}adamard diagonalizable graphs and quantum state transfer.
\newblock {\em Linear and Multilinear Algebra}, 73(17):3763--3790, 2025.

\bibitem[MN03]{MN03}
J.~J. Molitierno and M.~Neumann.
\newblock On trees with perfect matchings.
\newblock {\em Linear Algebra and its Applications}, 362:75--85, 2003.

\bibitem[MP14]{Nauty}
B.~D. McKay and A.~Piperno.
\newblock Practical graph isomorphism, {II}.
\newblock {\em Journal of Symbolic Computation}, 60:94--112, 2014.

\bibitem[Pap76]{Pap76}
C.~H. Papadimitriou.
\newblock The {NP}-completeness of the bandwidth minimization problem.
\newblock {\em Computing}, 16(3):263--270, 1976.

\bibitem[RT01]{ravelomanana2001asymptotic}
V.~Ravelomanana and L.~Thimonier.
\newblock Asymptotic enumeration of cographs.
\newblock {\em Electronic Notes in Discrete Mathematics}, 7:58--61, 2001.

\bibitem[Slo01a]{oeisA000065}
N.~J.~A. Sloane.
\newblock Sequence {A000065} in \emph{{T}he {O}n-{L}ine {E}ncyclopedia of {I}nteger {S}equences}.
\newblock \url{https://oeis.org/A000065}, 2001.
\newblock -1 + number of partitions of n. [Online; accessed June 18, 2026].

\bibitem[Slo01b]{oeisA025828}
N.~J.~A. Sloane.
\newblock Sequence {A025828} in \emph{{T}he {O}n-{L}ine {E}ncyclopedia of {I}nteger {S}equences}.
\newblock \url{https://oeis.org/A025828}, 2001.
\newblock Expansion of 1/((1-x\^{}3)*(1-x\^{}4)*(1-x\^{}6)). [Online; accessed June 22, 2026].

\bibitem[Sta97]{stanley2011enumerative}
R.~P. Stanley.
\newblock {\em Enumerative Combinatorics}, volume~1.
\newblock Cambridge University Press, 1997.

\bibitem[Var25]{Var25}
L.~M.~B. Varona.
\newblock {MatrixBandwidth.jl}: Fast algorithms for matrix bandwidth minimization and recognition.
\newblock {\em Journal of Open Source Software}, 10(116), 2025.
\newblock Article 9136.

\bibitem[Var26]{VaronaCode}
L.~M.~B. Varona.
\newblock Small $\{-1,0,1\}$-diagonalizable graphs: A computational survey.
\newblock Zenodo, \url{https://doi.org/10.5281/zenodo.21114288}, 2026.

\bibitem[Wan05]{wang2005survey}
L.~Wang.
\newblock A survey of results on integral trees and integral graphs.
\newblock 2005.

\end{thebibliography}
\appendix
\renewcommand{\thesection}{Appendix A}

\section{Computational results}

We now summarize the various computational results invoked throughout this paper.

\renewcommand{\thesection}{A}
\subsection{Balanced vectors}\label{app:balanced}

The set of all balanced vectors with positive integer entries (up to permutation of their entries) and sum equal to $n$ was computed for $n \in \{1,2,\ldots,9\}$ in \cite[Table~2]{johnston2025laplacian}. By instead checking balancedness via depth-first search (DFS), we can push this significantly further---Table~\ref{tab:balanced_smallsum} lists all such vectors with sum up to $13$, and Table~\ref{tab:totally_balanced_smallsum_counts} provides the number of such vectors with sum up to $73$.

\begin{table}[!htb]
    \centering
    \begin{tabular}{clc}\toprule
        $n$ & balanced vectors $\mathbf{v}$ & count \\\toprule
        $1$ & $(1)$ & $1$ \\\midrule
        $2$ & $(1^2)$, \ $(2)$ & $2$ \\\midrule
        $3$ & $(1^3)$, \ $(3)$ & $2$ \\\midrule
        $4$ & $(1^4)$, \ $(2,1^2)$, \ $(2^2)$, \ $(4)$ & $4$ \\\midrule
        $5$ & $(1^5)$, \ $(2,1^3)$, \ $(5)$ & $3$ \\\midrule
        $6$ & $(1^6)$, \ $(2,1^4)$, \ $(2^2,1^2)$, \ $(2^3)$, \ $(3,1^3)$, \ $(3^2)$, \ $(6)$ & $7$ \\\midrule
        $7$ & $(1^7)$, \ $(2,1^5)$, \ $(2^2,1^3)$, \ $(3,1^4)$, \ $(3,2,1^2)$, \ $(7)$ & $6$ \\\midrule
        $8$ & $(1^8)$, \ $(2,1^6)$, \ $(2^2,1^4)$, \ $(2^3,1^2)$, \ $(2^4)$, \ $(3,1^5)$, \ $(3,2,1^3)$, \ $(3,2^2,1)$, \ $(4,1^4)$, & $13$ \\
        & $(4,2,1^2)$, \ $(4,2^2)$, \ $(4^2)$, \ $(8)$ & \\\midrule
        $9$ & $(1^9)$, \ $(2,1^7)$, \ $(2^2,1^5)$, \ $(2^3,1^3)$, \ $(3,1^6)$, \ $(3,2,1^4)$, \ $(3,2^2,1^2)$, \ $(3^2,1^3)$, \ $(3^3)$, & $12$ \\
        & $(4,1^5)$, \ $(4,2,1^3)$, \ $(9)$ & \\\midrule
        $10$ & $(1^{10})$, \ $(2,1^8)$, \ $(2^2,1^6)$, \ $(2^3,1^4)$, \ $(2^4,1^2)$, \ $(2^5)$, \ $(3,1^7)$, \ $(3,2,1^5)$, \ $(3,2^2,1^3)$, & $22$ \\
        & $(3,2^3,1)$, \ $(3^2,1^4)$, \ $(3^2,2,1^2)$, \ $(4,1^6)$, \ $(4,2,1^4)$, \ $(4,2^2,1^2)$, \ $(4,2^3)$, \ $(4,3,1^3)$, & \\
        & $(4,3,2,1)$, \ $(5,1^5)$, \ $(5,2,1^3)$, \ $(5^2)$, \ $(10)$ & \\\midrule
        $11$ & $(1^{11})$, \ $(2,1^9)$, \ $(2^2,1^7)$, \ $(2^3,1^5)$, \ $(2^4,1^3)$, \ $(3,1^8)$, \ $(3,2,1^6)$, \ $(3,2^2,1^4)$, \ $(3,2^3,1^2)$, & $22$ \\
        & $(3^2,1^5)$, \ $(3^2,2,1^3)$, \ $(3^2,2^2,1)$, \ $(4,1^7)$, \ $(4,2,1^5)$, \ $(4,2^2,1^3)$, \ $(4,3,1^4)$, \ $(4,3,2,1^2)$, & \\
        & $(5,1^6)$, \ $(5,2,1^4)$, \ $(5,2^2,1^2)$, \ $(5,3,1^3)$, \ $(11)$ \\\midrule
        $12$ & $(1^{12})$, \ $(2,1^{10})$, \ $(2^2,1^8)$, \ $(2^3,1^6)$, \ $(2^4,1^4)$, \ $(2^5,1^2)$, \ $(2^6)$, \ $(3,1^9)$, \ $(3,2,1^7)$, & $45$ \\
        & $(3,2^2,1^5)$, \ $(3,2^3,1^3)$, \ $(3,2^4,1)$, \ $(3^2,1^6)$, \ $(3^2,2,1^4)$, \ $(3^2,2^2,1^2)$, \ $(3^2,2^3)$, \ $(3^3,1^3)$, & \\
        & $(3^4)$, \ $(4,1^8)$, \ $(4,2,1^6)$, \ $(4,2^2,1^4)$, \ $(4,2^3,1^2)$, \ $(4,2^4)$, \ $(4,3,1^5)$, \ $(4,3,2,1^3)$, \\
        & $(4,3,2^2,1)$, \ $(4,3^2,1^2)$, \ $(4^2,1^4)$, \ $(4^2,2,1^2)$, \ $(4^2,2^2)$, \ $(4^3)$, \ $(5,1^7)$, \ $(5,2,1^5)$, \\
        & $(5,2^2,1^3)$, \ $(5,2^3,1)$, \ $(5,3,1^4)$, \ $(5,3,2,1^2)$, \ $(6,1^6)$, \ $(6,2,1^4)$, \ $(6,2^2,1^2)$, \ $(6,2^3)$, \\
        & $(6,3,1^3)$, \ $(6,3^2)$, \ $(6^2)$, \ $(12)$ \\\midrule
        $13$ & $(1^{13})$, \ $(2,1^{11})$, \ $(2^2,1^9)$, \ $(2^3,1^7)$, \ $(2^4,1^5)$, \ $(2^5,1^3)$, \ $(3,1^{10})$, \ $(3,2,1^8)$, \ $(3,2^2,1^6)$, & $44$ \\
        & $(3,2^3,1^4)$, \ $(3,2^4,1^2)$, \ $(3^2,1^7)$, \ $(3^2,2,1^5)$, \ $(3^2,2^2,1^3)$, \ $(3^2,2^3,1)$, \ $(3^3,1^4)$, & \\
        & $(3^3,2,1^2)$, \ $(4,1^9)$, \ $(4,2,1^7)$, \ $(4,2^2,1^5)$, \ $(4,2^3,1^3)$, \ $(4,3,1^6)$, \ $(4,3,2,1^4)$, & \\
        & $(4,3,2^2,1^2)$, \ $(4,3^2,1^3)$, \ $(4,3^2,2,1)$, \ $(4^2,1^5)$, \ $(4^2,2,1^3)$, \ $(5,1^8)$, \ $(5,2,1^6)$, & \\
        & $(5,2^2,1^4)$, \ $(5,2^3,1^2)$, \ $(5,3,1^5)$, \ $(5,3,2,1^3)$, \ $(5,3,2^2,1)$, \ $(5,3^2,1^2)$, \ $(5,4,1^4)$, & \\
        & $(5,4,2,1^2)$, \ $(6,1^7)$, \ $(6,2,1^5)$, \ $(6,2^2,1^3)$, \ $(6,3,1^4)$, \ $(6,3,2,1^2)$, \ $(13)$ & \\\bottomrule
    \end{tabular}
    \caption{A list of all balanced vectors $\mathbf{v} \in \Zp^d$ whose sum is $n$, up to re-ordering of entries. The first $9$ rows of this table appeared as \cite[Table~2]{johnston2025laplacian}. The notation $a^b$ means that $a$ occurs $b$ times in the vector.}\label{tab:balanced_smallsum}
\end{table}

\begin{table}[!htb]
    \centering
    \begin{tabular}{cc}\toprule
        $n$ & count \\\toprule
        $14$ & $73$ \\
        $15$ & $83$ \\
        $16$ & $140$ \\
        $17$ & $145$ \\
        $18$ & $234$ \\
        $19$ & $260$ \\
        $20$ & $400$ \\
        $21$ & $444$ \\
        $22$ & $656$ \\
        $23$ & $739$ \\\bottomrule
    \end{tabular} \ \ \begin{tabular}{cc}\toprule
        $n$ & count \\\toprule
        $24$ & $1,091$ \\
        $25$ & $1,217$ \\
        $26$ & $1,711$ \\
        $27$ & $1,948$ \\
        $28$ & $2,727$ \\
        $29$ & $3,071$ \\
        $30$ & $4,187$ \\
        $31$ & $4,799$ \\
        $32$ & $6,456$ \\
        $33$ & $7,317$ \\\bottomrule
    \end{tabular} \ \ \begin{tabular}{cc}\toprule
        $n$ & count \\\toprule
        $34$ & $9,683$ \\
        $35$ & $11,014$ \\
        $36$ & $14,507$ \\
        $37$ & $16,527$ \\
        $38$ & $21,317$ \\
        $39$ & $24,279$ \\
        $40$ & $31,273$ \\
        $41$ & $35,519$ \\
        $42$ & $44,987$ \\
        $43$ & $51,519$ \\\bottomrule
    \end{tabular} \ \ \begin{tabular}{cc}\toprule
        $n$ & count \\\toprule
        $44$ & $64,717$ \\
        $45$ & $73,598$ \\
        $46$ & $92,045$ \\
        $47$ & $104,842$ \\
        $48$ & $130,037$ \\
        $49$ & $148,054$ \\
        $50$ & $181,916$ \\
        $51$ & $207,181$ \\
        $52$ & $253,936$ \\
        $53$ & $288,577$ \\\bottomrule
    \end{tabular} \ \ \begin{tabular}{cc}\toprule
        $n$ & count \\\toprule
        $54$ & $350,351$ \\
        $55$ & $399,180$ \\
        $56$ & $482,844$ \\
        $57$ & $548,174$ \\
        $58$ & $660,163$ \\
        $59$ & $749,814$ \\
        $60$ & $897,423$ \\
        $61$ & $1,020,280$ \\
        $62$ & $1,214,953$ \\
        $63$ & $1,377,849$ \\\bottomrule
    \end{tabular} \ \ \begin{tabular}{cc}\toprule
        $n$ & count \\\toprule
        $64$ & $1,638,683$ \\
        $65$ & $1,855,629$ \\
        $66$ & $2,193,863$ \\
        $67$ & $2,488,178$ \\
        $68$ & $2,932,196$ \\
        $69$ & $3,316,569$ \\
        $70$ & $3,896,378$ \\
        $71$ & $4,409,396$ \\
        $72$ & $5,159,822$ \\
        $73$ & $5,837,407$ \\\bottomrule
    \end{tabular}
    \caption{The number of balanced vectors $\mathbf{v} \in \Zp^d$ whose sum is $n$, up to re-ordering of entries. The values for $n \leq 13$ can be found in Table~\ref{tab:balanced_smallsum}.}\label{tab:totally_balanced_smallsum_counts}
\end{table}

We now describe this DFS. First, notice that a vector $\vv = (v_1,\ldots,v_d)$ is balanced if and only if the set
\[
    \mathrm{span}\big\{\mathbf{a} \in \{-1,0,1\}^d : \mathbf{a} \cdot \vv = 0 \big\}
\]
is $(d-1)$-dimensional. The members of $\big\{\mathbf{a} \in \{-1,0,1\}^d : \mathbf{a} \cdot \vv = 0 \big\}$ can be generated by DFS: choose the entries of $\mathbf{a} = (a_1,\ldots,a_d)$ one at a time from $\{-1,0,1\}$, keeping track of the partial dot product
\[
    a_1v_1+\cdots+a_kv_k.
\]
A branch can be pruned whenever the partial sum is large enough in absolute value that the remaining coordinates cannot bring it back to $0$. For each full row $\mathbf{a}$ with $\mathbf{a} \cdot \vv = 0$, add $\mathbf{a}$ to a running row-reduction computation over $\mathbb{Q}$. If the rank reaches $d-1$, then $\vv$ is balanced; otherwise, $\vv$ is not balanced.

\subsection{Enumerating small Laplacian integral and $\{-1,0,1\}$-diagonalizable graphs}\label{sec:appendix_compute_lapint}

Suppose that $G$ is a graph on $n$ vertices. Our algorithm to determine whether $G$ is $\{-1,0,1\}$-diagonalizable begins by computing the distinct eigenvalues $0 = \lambda_1, \lambda_2, \ldots, \lambda_r$ of $L(G)$, terminating early and returning false if any eigenvalue is not an integer. If $G$ passes this Laplacian integrality check, we then seek to construct matrices $W_1, W_2, \ldots, W_r$ such that the columns of each $W_i$ constitute precisely the set of all $\{-1,0,1\}$-vectors (up to linear span) in the $\lambda_i$-eigenspace of $L(G)$.

If $G$ is connected, then its kernel is spanned by the all-ones vector, so we can simply set $W_1 = \mathbf{1}$. Otherwise, we iterate over every $\{-1,0,1\}$-vector $\vv \in \R^n \setminus \{\mathbf{0}\}$ whose first nonzero entry is equal to $1$, then check whether $L(G)\vv = \mathbf{0}$, and add $\vv$ as a column to $W_1$ if so (there are $(3^n - 1) / 2$ such vectors). To aid in computing the remaining matrices $W_2, W_3, \ldots, W_r$, we observe that every non-kernel eigenvector must be orthogonal to the kernel (which contains $\mathbf{1}$ whether $G$ is connected or not). We therefore iterate over every $\{-1,0,1\}$-vector $\mathbf{w} \in \R^n \setminus \{\mathbf{0}\}$ whose first nonzero entry is equal to $1$ and which has an equal number of $-1$s and $1$s. We then check whether $L(G)\mathbf{w} = \lambda_i\mathbf{w}$ for any $i \ge 2$ and add $\mathbf{w}$ as a column to $W_i$ if so. (This orthogonality filter reduces the number of such vectors to $\frac{1}{2} \sum_{k = 1}^{\lfloor n / 2 \rfloor} \binom{n}{k} \binom{n - k}{k}$, which is in $\Theta(3^n / \sqrt{n})$ as opposed to the na\"ive count of $(3^n - 1) / 2 \in \Theta(3^n)$.)

Once this is done, we iterate over $W_1, W_2, \ldots, W_r$ and compute a rank-revealing QR (or RRQR) factorization of each. If the rank of any $W_i$ is found to be less than the multiplicity $\mu_i$ of the eigenvalue $\lambda_i$, then no associated $\{-1,0,1\}$-eigenbasis exists (and thus $G$ is not $\{-1,0,1\}$-diagonalizable), so we return false. Otherwise, we use the column pivots of each RRQR factorization to find exactly $\mu_i$ $\{-1,0,1\}$-eigenvectors associated with $\lambda_i$, showing that $G$ is indeed $\{-1,0,1\}$-diagonalizable. In this case, we return true along with a diagonalizing $\{-1,0,1\}$-matrix. Once we combine this methodology with the geng and genbg programs (both part of the nauty package \cite{Nauty}) to iterate over all connected, connected regular, and connected bipartite regular graphs of a given order, we obtain lists of all small Laplacian integral and $\{-1,0,1\}$-diagonalizable graphs; see Tables~\ref{tab:bipartite_negone}, \ref{tab:graph_counts_thirteen}, \ref{tab:graph_counts_regular}, and~\ref{tab:graph_counts_regular_disconnected}, as well as \cite{VaronaCode}.

\begin{table}[!htb]
    \centering
    \begin{tabular}{c|cccc|cccc}\toprule
        $n$ & $l(n)$ & $cl(n)$ & $rl(n)$ & $crl(n)$ & $s(n)$ & $cs(n)$ & $rs(n)$ & $crs(n)$ \\\toprule
        $1$ & $1$ & $1$ & $1$ & $1$ & $1$ & $1$ & $1$ & $1$\\
        $2$ & $2$ & $1$ & $2$ & $1$ & $2$ & $1$ & $2$ & $1$ \\
        $3$ & $4$ & $2$ & $2$ & $1$ & $3$ & $1$ & $2$ & $1$ \\
        $4$ & $10$ & $5$ & $4$ & $2$ & $7$ & $3$ & $4$ & $2$ \\
        $5$ & $24$ & $12$ & $2$ & $1$ & $10$ & $2$ & $2$ & $1$ \\
        $6$ & $70$ & $37$ & $8$ & $5$ & $23$ & $8$ & $8$ & $5$ \\
        $7$ & $188$ & $94$ & $4$ & $2$ & $34$ & $5$ & $3$ & $1$ \\
        $8$ & $553$ & $280$ & $10$ & $6$ & $81$ & $26$ & $10$ & $6$ \\
        $9$ & $1,721$ & $912$ & $10$ & $7$ & $123$ & $16$ & $6$ & $3$ \\
        $10$ & $5,716$ & $3,164$ & $22$ & $15$ & $250$ & $45$ & $14$ & $7$ \\
        $11$ & $16,848$ & $8,424$ & $6$ & $3$ & $392$ & $35$ & $4$ & $1$ \\
        $12$ & $62,052$ & $35,883$ & $74$ & $60$ & $947$ & $263$ & $59$ & $46$ \\
        $13$ & $183,836$ & $91,918$ & $12$ & $6$ & $1,445$ & $124$ & $6$ & $1$ \\
        $14$ & ? & ? & $156$ & $136$ & ? & ? & $19$ & $5$ \\
        $15$ & ? & ? & $168$ & $154$ & ? & ? & $25$ & $16$ \\\bottomrule
    \end{tabular}
    \caption{The number of Laplacian integral and $\{-1,0,1\}$-diagonalizable graphs that are not necessarily connected, are connected, are regular, or are connected and regular, for small values of $n$. Extensions of this table for prime $n$ are provided in Tables~\ref{tab:prime_101_all} and~\ref{tab:regular_prime_101}.}\label{tab:graph_counts_thirteen}
\end{table}

\begin{table}[!htb]
    \centering
    \begin{tabular}{c|cccccccccccccccccc|c}\toprule
        $n$ / $r$ & $1$ & $2$ & $3$ & $4$ & $5$ & $6$ & $7$ & $8$ & $9$ & $10$ & $11$ & $12$ & $13$ & $14$ & $15$ & $16$ & $17$ & $18$ & Total \\\toprule
        $2$ & $1$ & -- & -- & -- & -- & -- & -- & -- & -- & -- & -- & -- & -- & -- & -- & -- & -- & -- & $1$ \\
        $3$ & $0$ & $1$ & -- & -- & -- & -- & -- & -- & -- & -- & -- & -- & -- & -- & -- & -- & -- & -- & $1$ \\
        $4$ & $0$ & $1$ & $1$ & -- & -- & -- & -- & -- & -- & -- & -- & -- & -- & -- & -- & -- & -- & -- & $2$ \\
        $5$ & $0$ & $0$ & $0$ & $1$ & -- & -- & -- & -- & -- & -- & -- & -- & -- & -- & -- & -- & -- & -- & $1$ \\
        $6$ & $0$ & $1$ & $2$ & $1$ & $1$ & -- & -- & -- & -- & -- & -- & -- & -- & -- & -- & -- & -- & -- & $5$ \\
        $7$ & $0$ & $0$ & $0$ & $1$ & $0$ & $1$ & -- & -- & -- & -- & -- & -- & -- & -- & -- & -- & -- & -- & $2$ \\
        $8$ & $0$ & $0$ & $1$ & $2$ & $1$ & $1$ & $1$ & -- & -- & -- & -- & -- & -- & -- & -- & -- & -- & -- & $6$ \\
        $9$ & $0$ & $0$ & $0$ & $4$ & $0$ & $2$ & $0$ & $1$ & -- & -- & -- & -- & -- & -- & -- & -- & -- & -- & $7$ \\
        $10$ & $0$ & $0$ & $3$ & $1$ & $2$ & $5$ & $2$ & $1$ & $1$ & -- & -- & -- & -- & -- & -- & -- & -- & -- & $15$ \\
        $11$ & $0$ & $0$ & $0$ & $0$ & $0$ & $1$ & $0$ & $1$ & $0$ & $1$ & -- & -- & -- & -- & -- & -- & -- & -- & $3$ \\
        $12$ & $0$ & $0$ & $2$ & $8$ & $13$ & $14$ & $10$ & $7$ & $4$ & $1$ & $1$ & -- & -- & -- & -- & -- & -- & -- & $60$ \\
        $13$ & $0$ & $0$ & $0$ & $0$ & $0$ & $0$ & $0$ & $3$ & $0$ & $2$ & $0$ & $1$ & -- & -- & -- & -- & -- & -- & $6$ \\
        $14$ & $0$ & $0$ & $0$ & $7$ & $19$ & $32$ & $33$ & $20$ & $14$ & $7$ & $2$ & $1$ & $1$ & -- & -- & -- & -- & -- & $136$ \\
        $15$ & $0$ & $0$ & $0$ & $12$ & $0$ & $58$ & $0$ & $59$ & $0$ & $20$ & $0$ & $4$ & $0$ & $1$ & -- & -- & -- & -- & $154$ \\
        $16$ & $0$ & $0$ & $0$ & $2$ & $33$ & ? & ? & ? & ? & $36$ & $11$ & $14$ & $4$ & $1$ & $1$ & -- & -- & -- & ? \\
        $17$ & $0$ & $0$ & $0$ & $0$ & $0$ & $0$ & $0$ & $0$ & $0$ & $7$ & $0$ & $19$ & $0$ & $2$ & $0$ & $1$ & -- & -- & $29$ \\
        $18$ & $0$ & $0$ & $0$ & $26$ & ? & ? & ? & ? & ? & ? & ? & ? & $50$ & $18$ & $6$ & $1$ & $1$ & -- & ? \\
        $19$ & $0$ & $0$ & $0$ & $0$ & $0$ & $0$ & $0$ & $0$ & $0$ & $1$ & $0$ & $25$ & $0$ & $27$ & $0$ & $4$ & $0$ & $1$ & $58$ \\\bottomrule
    \end{tabular}
    \caption{The number $crl(n,r)$ of connected $r$-regular $n$-vertex integral graphs. Other known values include $crl(n,r) = 0$ when $r \leq 2$ and $n \geq 7$, $crl(20,3) = 2$, $crl(24,3) = crl(30,3) = 1$ and $crl(n,3) = 0$ otherwise, and $crl(23,12) = 1$, $crl(23,14) = 30$, $crl(23,18) = 120$, $crl(23,20) = 4$ and $crl(23,22) = 1$.}\label{tab:graph_counts_regular}
\end{table}

\begin{table}[!htb]
    \centering
    \begin{tabular}{c|cccccccccccccccccc|c}\toprule
        $n$ / $r$ & $1$ & $2$ & $3$ & $4$ & $5$ & $6$ & $7$ & $8$ & $9$ & $10$ & $11$ & $12$ & $13$ & $14$ & $15$ & $16$ & $17$ & $18$ & Total \\\toprule
        $2$ & $1$ & -- & -- & -- & -- & -- & -- & -- & -- & -- & -- & -- & -- & -- & -- & -- & -- & -- & $1$ \\
        $3$ & $0$ & $1$ & -- & -- & -- & -- & -- & -- & -- & -- & -- & -- & -- & -- & -- & -- & -- & -- & $1$ \\
        $4$ & $1$ & $1$ & $1$ & -- & -- & -- & -- & -- & -- & -- & -- & -- & -- & -- & -- & -- & -- & -- & $3$ \\
        $5$ & $0$ & $0$ & $0$ & $1$ & -- & -- & -- & -- & -- & -- & -- & -- & -- & -- & -- & -- & -- & -- & $1$ \\
        $6$ & $1$ & $2$ & $2$ & $1$ & $1$ & -- & -- & -- & -- & -- & -- & -- & -- & -- & -- & -- & -- & -- & $7$ \\
        $7$ & $0$ & $1$ & $0$ & $1$ & $0$ & $1$ & -- & -- & -- & -- & -- & -- & -- & -- & -- & -- & -- & -- & $3$ \\
        $8$ & $1$ & $1$ & $2$ & $2$ & $1$ & $1$ & $1$ & -- & -- & -- & -- & -- & -- & -- & -- & -- & -- & -- & $9$ \\
        $9$ & $0$ & $2$ & $0$ & $4$ & $0$ & $2$ & $0$ & $1$ & -- & -- & -- & -- & -- & -- & -- & -- & -- & -- & $9$ \\
        $10$ & $1$ & $2$ & $5$ & $2$ & $2$ & $5$ & $2$ & $1$ & $1$ & -- & -- & -- & -- & -- & -- & -- & -- & -- & $21$ \\
        $11$ & $0$ & $1$ & $0$ & $1$ & $0$ & $1$ & $0$ & $1$ & $0$ & $1$ & -- & -- & -- & -- & -- & -- & -- & -- & $5$ \\
        $12$ & $1$ & $4$ & $7$ & $10$ & $14$ & $14$ & $10$ & $7$ & $4$ & $1$ & $1$ & -- & -- & -- & -- & -- & -- & -- & $73$ \\
        $13$ & $0$ & $2$ & $0$ & $3$ & $0$ & $0$ & $0$ & $3$ & $0$ & $2$ & $0$ & $1$ & -- & -- & -- & -- & -- & -- & $11$ \\
        $14$ & $1$ & $2$ & $7$ & $14$ & $20$ & $33$ & $33$ & $20$ & $14$ & $7$ & $2$ & $1$ & $1$ & -- & -- & -- & -- & -- & $155$ \\
        $15$ & $0$ & $4$ & $0$ & $20$ & $0$ & $59$ & $0$ & $59$ & $0$ & $20$ & $0$ & $4$ & $0$ & $1$ & -- & -- & -- & -- & $167$ \\
        $16$ & $1$ & $4$ & $14$ & $11$ & $36$ & ? & ? & ? & ? & $36$ & $11$ & $14$ & $4$ & $1$ & $1$ & -- & -- & -- & ? \\
        $17$ & $0$ & $2$ & $0$ & $19$ & $0$ & $7$ & $0$ & $0$ & $0$ & $7$ & $0$ & $19$ & $0$ & $2$ & $0$ & $1$ & -- & -- & $57$ \\
        $18$ & $1$ & $6$ & $18$ & $50$ & ? & ? & ? & ? & ? & ? & ? & ? & $50$ & $18$ & $6$ & $1$ & $1$ & -- & ? \\
        $19$ & $0$ & $4$ & $0$ & $27$ & $0$ & $25$ & $0$ & $1$ & $0$ & $1$ & $0$ & $25$ & $0$ & $27$ & $0$ & $4$ & $0$ & $1$ & $115$ \\\bottomrule
    \end{tabular}
    \caption{The number $rl(n,r)$ of (not necessarily connected) $r$-regular $n$-vertex integral graphs.}\label{tab:graph_counts_regular_disconnected}
\end{table}

It is worth noting that much of the data archived at \cite{VaronaCode} also records a property known as \emph{$\{-1,0,1\}$-bandwidth}, which we describe now. First, recall that the \emph{bandwidth} of an $m \times m$ matrix $X$ is the minimum integer $b \in \{0, 1, \ldots, m - 1\}$ such that $[P^{\mathrm{T}}XP]_{i,j} = 0$ for all $i, j$ with $\lvert i - j \rvert > b$ for some permutation matrix $P$. If $L(G)$ is not $\{-1,0,1\}$-diagonalizable, then the $\{-1,0,1\}$-bandwidth of $G$ is simply $\infty$; otherwise, the $\{-1,0,1\}$-bandwidth of $G$ is the minimum integer $k \in \{1, 2, \ldots, n\}$ such that there exists some $\{-1,0,1\}$-matrix $W$ that diagonalizes $L(G)$ and whose Gram matrix $W^{\mathrm{T}}W$ has bandwidth $k - 1$ \cite{johnston2025laplacian}. (This unfortunate off-by-one mismatch between matrix bandwidth and $\{-1,0,1\}$-bandwidth arises from different indexing conventions in computer science and mathematics.) By orthogonality of the eigenspaces of $L(G)$, this condition is equivalent to there existing, for each $i$, some $n \times \mu_i$ $\{-1,0,1\}$-matrix $B_i$ whose columns form a $\lambda_i$-eigenbasis and whose Gram matrix $B_i^{\mathrm{T}}B_i$ has bandwidth $\le k - 1$---then we simply set $W = \begin{bmatrix}
    B_1 & B_2 & \cdots & B_r
\end{bmatrix}$.

To compute the $\{-1,0,1\}$-bandwidth of $G$, we begin in the same manner as we do when checking if $G$ is $\{-1,0,1\}$-diagonalizable: by constructing all $W_i$'s and checking that they each have rank equal to $\mu_i$. Upon confirming $\{-1,0,1\}$-diagonalizability in this manner, we then determine, for incrementing values of $k$ starting at $1$, whether $G$ has $\{-1,0,1\}$-bandwidth at most $k$ via the following procedure. For each $W_i$, we perform a DFS of all column subsets of $W_i$, seeking a subset that forms an $n \times \mu_i$ matrix $B_i$ whose Gram matrix $B_i^{\mathrm{T}}B_i$ has bandwidth $\le k - 1$. We do this by starting with the empty set as a base case, then recursively adding columns and performing both a linear independence check and a ``Gramian bandwidth $\le k - 1$'' check at every node of the search tree, backtracking in the standard DFS manner whenever either of these conditions is violated. (Computing matrix bandwidth is known to be NP-complete \cite{Pap76}, so we use a pre-existing, highly optimized Julia implementation of \cite{CSG05}'s algorithm made available by \cite{Var25}.) Once we have succeeded in finding a valid subset for every eigenspace, we stop incrementing and return the current value of $k$ along with the corresponding diagonalizing $\{-1,0,1\}$-matrix.

The nature of our $\{-1,0,1\}$-bandwidth algorithm makes it untenable when any eigenspace either is too large or has too many $\{-1,0,1\}$-eigenvectors in it, even for some graphs on as few as $16$ vertices. (Both cases result in a combinatorial explosion of the number of cardinality-$\mu_i$ column subsets of the corresponding $W_i$ to check.) Additionally, even our $\{-1,0,1\}$-diagonalizability algorithm becomes infeasible when the size of the input graph exceeds $19$ or so because of the $\Theta(3^n / \sqrt{n})$ (or $\Theta(3^n)$ if the graph is disconnected) growth of the number of potential $\{-1,0,1\}$-eigenvectors to check.

The Julia code that produced these results, alongside the data itself (stored in Apache Arrow format), can be found at \cite{VaronaCode}. All computations were run on the Nibi supercomputing cluster at the University of Waterloo, with access provided by Matthew Betti and the Digital Research Alliance of Canada.

\end{document}